\documentclass[11pt]{amsart}
\usepackage[top=2.5 cm,bottom=2 cm,left=1.5 cm,right=1.5 cm]{geometry}

\usepackage{amssymb,amsfonts,graphicx,amsmath,amsthm,amstext,latexsym,amsfonts}
\usepackage{graphicx,epsfig}
\usepackage{bbm}
\usepackage{t1enc}
\usepackage{mathrsfs}
\usepackage{subfig}
\usepackage{hyperref}
\usepackage{tikz}
\usetikzlibrary{intersections}
\theoremstyle{plain}

\numberwithin{equation}{section}
\newtheorem{theorem}{Theorem}[section]
\newtheorem{proposition}[theorem]{Proposition}
\newtheorem{corollary}[theorem]{Corollary}
\newtheorem{lemma}[theorem]{Lemma}
\newtheorem{definition}[theorem]{Definition}
\newtheorem{example}[theorem]{Example}
\newtheorem{remark}[theorem]{Remark}

\usepackage{color}
\definecolor{darkred}{rgb}{0.8,0,0}
\definecolor{darkblue}{rgb}{0,0,0.7}
\definecolor{darkgreen}{rgb}{0,0.4,0}

\newcommand{\eps}{\varepsilon}

\newcommand{\R}{{\mathbb R}}

\newcommand{\W}{{\mathcal W}}

\newcommand{\FF}{{\mathcal F}}
\newcommand{\EE}{{\mathcal E}}

\newcommand{\V}{{\mathcal V}}

\newcommand{\un}{{\rm 1\kern -2.5pt l}}
\newcommand{\tr}{{\rm Tr}}

\def\yy{\mathbf{y}}
\def\n{\mathbf{n}}

\def\xx{\mathbf{x}}
\def\aa{\mathbf{a}}
\def\bb{\mathbf{b}}

\def\vv{\mathbf{v}}
\def\zz{\mathbf{z}}

\def\eps{\varepsilon}
\def\R{{\mathbb R}}
\def\N{{\mathcal N}}

\def\M{{\mathcal M}}

\def\H{{\mathcal H}}

\def\eps{\varepsilon}
\def\R{{\mathbb R}}
\def\N{{\mathbb N}}

\def\e{{\mathcal F}}

\def\M{{\mathcal M}}

\def\H{{\mathcal H}}

\def\F{{\mathcal F}}

\def\E{{\mathbb{E}}}

\def\argmin{\mathop{{\rm argmin}}\nolimits}

\def\Tr{\mathop{{\rm Tr}}\nolimits}

\def\dv{\mathop{{\rm div}}\nolimits}

\def\sym{\mathop{{\rm sym}}\nolimits}

\def\span{\mathop{{\rm span}}\nolimits}

\def\v{\mathbf{v}}
\def\e{\mathbf{E}}
\def\z{\mathbf{z}}
\def\v{{\bf v}}
\def\w{{\bf w}}
\def\e{{\bf e}}
\def\x{{\bf x}}

\def\Rot{\mathbf{R}}
\def\Id{\mathbf{I}}

\def\wconv{\rightharpoonup}

\renewcommand{\epsilon}{\varepsilon}
\newcommand{\beeq}{\begin{equation}}
\newcommand{\eneq}{\end{equation}}
\newcommand{\bear}{\begin{array}}
\newcommand{\enar}{\end{array}}
\newcommand{\bema}{\begin{displaymath}}
\newcommand{\enma}{\end{displaymath}}

\newcommand{\beea}{\begin{eqnarray}}
\newcommand{\enea}{\end{eqnarray}}

\newcommand{\om}{\Omega}

\def \qed{ \hfill \rule{6pt}{6pt}
\medskip}

\newcommand{\lab}[1]{ \label{#1} }

\def\Rot{\mathbf{R}}
\def\Id{\mathbf{I}}

\def\wconv{\rightharpoonup}

\title[]{The gap
in Pure Traction Problems
between Linear Elasticity and Variational Limit of
Finite Elasticity
}
   \author[]{Francesco Maddalena, Danilo Percivale, Franco Tomarelli }

 \address{Politecnico di Bari, Dipartimento di Meccanica, Matematica, Management, via Re David 200, 70125 Bari, Italy}
 \email{francesco.maddalena@poliba.it}
 \address{Universit\`{a} di Genova, Dipartimento di Ingegneria Meccanica,
  Piazzale Kennedy, Fiera del Mare, Padiglione D, 16129 Genova, Italy}
  \email{percivale@diptem.unige.it}
\address{Politecnico di Milano, Dipartimento di Matematica,  Piazza Leonardo da Vinci 32, 20133 Milano, Italy}
\email{franco.tomarelli@polimi.it}
\date{\today}  \subjclass{}
\begin{document}
 \maketitle
\begin{abstract}
A limit elastic energy for pure traction problem is derived from re-scaled nonlinear energy
of an hyperelastic material body 
subject to an equilibrated force field.\\
We show that the strains of minimizing sequences associated to re-scaled non linear energies
weakly converge,
up to subsequences, to the strains of minimizers of a limit energy,
provided an additional compatibility condition is fulfilled by the force field.
\\
The limit energy is 
different from classical energy of linear elasticity;
nevertheless the compatibility condition entails the coincidence of related minima and minimizers.\\
A strong violation of this condition provides a limit energy which is unbounded from below, while a mild violation
may produce 
a limit energy with infinitely many extra minimizers which are not minimizers of standard linear elastic energy and whose strains are not uniformly bounded.
\\
A relevant consequence of this analysis is that
a rigorous validation of linear elasticity fails for compressive force fields that do not fulfil such compatibility condition.
\end{abstract}

\tableofcontents
\begin{flushleft}
  {\bf AMS Classification Numbers (2010):\,} 49J45, 74K30, 74K35, 74R10.\\
  {\bf Key Words:\,} Calculus of Variations, Pure Traction problems, Linear Elasticity, Nonlinear Elasticity, Finite Elasticity,
  Critical points, Gamma-convergence, Asymptotic analysis, nonlinear Neumann problems.\\
\end{flushleft}

\section{Introduction}
{\it Linear Theory of Elasticity} (\cite{Gu})
plays a relevant role among Mathematical-Physics theories
for clearness, rigorous mathematical status and persistence.
It was a great achievement of last centuries, which inspired many other theories of Continuum Mechanics
and led to the formulation of a more general theory named {\it Nonlinear Elasticity} (\cite{Lo},\cite{TN}), also known as
{\it Finite Elasticity} to underline that no smallness assumptions is required.\\
There was always agreement in scholars community that the relation between
linear and nonlinear theory amounts to the linearization of strain measure
under the assumption of small displacement gradients:
this is the precondition advocated in almost all the texts on elasticity.
Nevertheless only at the beginning of present century, when appropriated tools of mathematical analysis were suitably tuned, the problem of a rigorous deduction of any particular theory based on some
{\it approximation hypotheses} from a more general {\it exact} theory, became a scientific issue related to general problem of {\it validation} of a theory, as explained in \cite{PPG}).\\
In this conceptual framework, G. Dal Maso, M. Negri and D. Percivale in \cite{DMPN} proved that
problems ruled by linear elastic energies can be rigorously deduced from problems ruled by non-linear energies
in the case of \textsl{Dirichlet and mixed boundary conditions}, by exploiting De Giorgi $\Gamma$-convergence theory (\cite{DM},\cite{DGFra}).
This result clarified the mathematical consistency of the
linear boundary value problems, under displacements and forces prescribed on the boundary of a three-dimensional material body, via a rigorous deduction from the nonlinear elasticity theory.
We mention several papers facing issues in elasticity which are connected with the context of our paper:
\cite{ABK},\,\cite{ADMDS},\,\cite{ADMLP}, \cite{ABP},\,\cite{AP}\,\cite{BBGT},\,\cite{BD},\,\cite{BT},\,\cite{CLT},\,\cite{LM}\,\cite{MPT},
\!\cite{MPT2},\,\cite{MPTFvK},\,\cite{PT1},\,\cite{PTplate},\,\cite{PT2},\,\cite{PT4}.
\\
The present paper focus on the same general question  studied in \cite{DMPN}, but here we deal with the
\textsl{pure traction problem}, i.e. the case where the elastic body is subject to a system of equilibrated forces and no Dirichlet condition is imposed on the boundary.
\vskip0.1cm
Let an open set $\Omega\subset\R^N,\ N=2,3,$ be the reference configuration of an hyperelastic material body, then the stored energy due to a deformation  $\mathbf y$ can be  expressed as
a functional of the deformation gradient $\nabla\yy$
\begin{equation*}
\int_\Omega\mathcal \,\mathcal W(\xx, \nabla\yy)\,d\xx
\end{equation*}
where  $ \W:\om \times\! \M^{N \times N}\!\to \! [0, +\infty ]$ is a frame indifferent function, $\M^{N\times N}$ is the set of real $N\times N$ matrices and
$\mathcal W(\xx, \mathbf F)<+\infty$ if and only if $\det \mathbf F >0$.\\ Then due to frame indifference there exists a function $\mathcal V$ such that
\begin{equation*}
\W(\xx,\mathbf F)=\V(\xx,\textstyle{\frac{1}{2}}( \mathbf F^T \mathbf F - \mathbf I))\,,
\qquad
\ \forall\, \mathbf F\in \M^{N\times N},\  \hbox{ a.e. }\xx\in \om.
\end{equation*}
We set $\mathbf F=\Id +h\mathbf B$, where $h> 0$ is an adimensional small parameter and
\begin{equation*}
\mathcal V_{h}(\xx,\mathbf B):= h^{-2}\mathcal W(\xx,\Id+h\mathbf B).
\end{equation*}
We assume that the reference configuration has zero energy and is stress free, i.e. $$\W(\xx,\mathbf I)=0,
\quad D \W(\xx,\mathbf I)=\mathbf 0 \quad \hbox{for a.e. }\xx \in \om \,,$$ and that $\W$ is regular enough in the second variable.
Then Taylor's formula entails
\begin{equation*}
\mathcal V_{h}(\xx,\mathbf B)=\mathcal V_{0}(\xx, \sym \mathbf B) +o(1) \qquad \hbox{ as } h \to 0_+
\end{equation*}
where $\sym \mathbf B:=\frac{1}{2}(\mathbf B^{T}+\mathbf B)$ and
\begin{equation*}
\mathcal V_{0}(\xx, \sym \mathbf B):=\frac{1}{2} \sym \mathbf B\, D^{2}\mathcal V(\xx,\mathbf 0) \,\sym \mathbf B.
\end{equation*}
If the deformation $\yy$ is close to the identity up to a small displacement, say $\mathbf y(\xx)= \xx+h\mathbf \v(\xx)$ with
bounded $\nabla \v\,$ then,
by setting $\,\E(\v):= \textstyle\frac{1}{2}(\nabla\v^{T}+\nabla\v)$\,, one plainly obtains
\begin{equation}\lab{linpunt}
\lim_{h\to 0}\int_{\om}\mathcal V_{h}(\xx,\nabla\v)\,d\xx= \int_{\om}\mathcal V_{0}(\xx,\E(\v))\,d\xx
\end{equation}
The  relationship \eqref{linpunt} was considered as the main justification of the linearized theory of elasticity,
but such point-wise convergence does not even entail that minimizers fulfilling a given fixed
Dirichlet boundary condition actually converge to the minimizers of the corresponding limit boundary value problem:
this phenomenon is made explicit by the Example 3.5 in \cite{DMPN} which exhibits a lack of compactness when $\mathcal V$
has several minima.\\
To set the Dirichlet problem in a variational perspective, referring to a prescribed vector field
$\v_{0}\in W^{1,\infty}(\om,\R^N)$ as the boundary condition on a given closed subset $\Sigma$ of
$\partial \om$ with ${\mathcal H}^{N-1} (\Sigma)>0$ and to a given load $ \mathbf g \in L^2(\om,\R^N)$,
one has to study the asymptotic behavior of the sequence of functionals $\mathcal I_h$, which is defined as
\begin{equation*}
\displaystyle \mathcal I_h(\v)=\left\{\begin{array}{ll}
\displaystyle \int_\Omega\mathcal V_{h}(\xx, \nabla \v)\,d\xx-\int_{\Omega}\mathbf g\cdot\v
\,d\xx & \ {\rm if }
\ \v\in H^1_{\v_{0},\Sigma}
\vspace{0.2cm}
\\
+\infty
& \ \hbox{\rm  else in\ }\, H^1(\Omega),\R^N),\\
\end{array} \right.
\end{equation*}
where 
$H^1_{\v_{0},\Sigma}$ denotes the closure in $H^1(\om,\R^N)$
of the space of displacements $\v \in W^{1,\infty}(\om,\R^N)$
such that $\v=\v_{0}$ on $\Sigma$: it was proved in \cite{DMPN} that (under natural growth conditions and suitable regularity hypotheses on $\mathcal W$) every sequence $\v_{h}$ fulfilling
\bema
\mathcal I_{h}(\v_{h})=
\inf \mathcal I_h  + o(1)\,
\enma
has a subsequence converging weakly
in $H^1(\om,\R^N)$
to the (unique) minimizer $\v_*$ of
the functional $\mathcal I$ representing the total energy in linear elasticity, e.g.
\begin{equation*}
\displaystyle \mathcal I(\v)=\left\{\,\begin{array}{ll}
\displaystyle \int_\Omega\mathcal \,
\mathcal V_0\left(\xx,\mathbb E(\vv)\right)
\,d\xx\,-\,\int_{\partial \Omega}\mathbf g\cdot\v
\ d\H^{n-1}(\xx) & \ {\rm if }\:\:  \v\in H^1_{\v_{0},\Sigma}
\vspace{0.1cm}
\\
+\infty
& \ \hbox{\rm  else in\ }\, H^1(\Omega;\mathbf R^{N})\, ,\\
\end{array} \right.
\end{equation*}
%
and that the re-scaled energies converge, say
\begin{equation*} \label{minenerg}
\lim_{h\to 0}
\mathcal I_{h}(\v_{h})\ = \ \mathcal I (  \vv_*)\ =\
    \int_{\om}
    \mathcal V_0\left(\xx,\mathbb E(\vv_*)\right)
    \,d\xx
  - \int_\om \mathbf g\cdot \v_*\,d\xx \,.
\end{equation*}
Such result  represents a complete variational justification of linearized
elasticity, at least as far as Dirichlet and mixed boundary value problems are concerned.
So it is natural to ask whether a similar result holds true also for pure traction problems whose variational formulation is described below.
\vskip0.2cm
In the present paper we focus our analysis on Neumann boundary conditions, say the pure traction problem in elasticity. 
Precisely we assume that $\mathbf f\in L^2(\partial \Omega;\R^N),\ \mathbf g\in L^{2}(\Omega;\R^N)$ are respectively the prescribed boundary and body force fields, such that the whole system of forces is equilibrated, namely
\begin{equation}
\label{globalequiintro}
\mathcal L(\v)\ :=\
\int_{\partial\Omega} \mathbf f\cdot\mathbf z\,d\H^{N-1}+\int_{\Omega} \mathbf g\cdot\mathbf z\,d\xx\,=\,0
\qquad \forall \mathbf z \,:\  \mathbb E(\mathbf z)\equiv \mathbf 0\,
\end{equation}
and we consider the sequence of energy functionals
\begin{equation}
\displaystyle \F_h(\v)\ =\
\displaystyle \int_\Omega\mathcal V_{h}(\xx, \nabla \v)d\xx-\mathcal L(\v)\,.
 \end{equation}
We inquire whether  the asymptotic relationship $\F_{h}(\v_{h})=\inf \F_{h}+o(1)$ as $h\to 0_+$ implies, up to subsequences, some kind of
weak convergence of  $\v_{h}$ to a minimizer $\v_{0}$ of a suitable limit functional in $H^{1}(\om;\mathbf R^{N})$.
\\
We emphasize that in the case of Neumann condition on the whole boundary things are not so plain:
indeed even by choosing $\Omega$ Lipschitz and assuming the simplest dependance of the stored energy density $\mathcal W$
on the deformation gradient $\mathbf F$, say (see \cite{C})
\begin{equation}
\lab{WquadIntrod}
\W(\xx,\mathbf F)=\left\{\begin{array}{ll}
 |\mathbf F^{T}\mathbf F-\Id|^{2}\  &\hbox{if} \ \det \mathbf F>0\\
\vspace{0.1cm}
&\\
 +\infty\ &\hbox{otherwise,}\\
 \end{array}\right.
\end{equation}
if $\mathbf g\equiv 0$,
$\ \mathbf f=f\n, \ f<0$ and $\n$ denots the outward normal to $\partial\om$
(so that the global condition \eqref{globalequiintro} holds true) then by the same techniques of \cite{DMPN} one can exhibit the $\Gamma$-limit of $\F_h$ with respect to
weak $H^1$ topology:
\begin{equation}
\Gamma(w \, H^1)\lim_{h\rightarrow 0}\ \F_h(\v)\ =\mathcal E(\v)\,,
\end{equation}
where
\begin{equation}\label{Eclassic}
\mathcal E(\v)=4\int_\Omega\vert \E(\v)\vert^2\,d\xx-f\int_{\partial\Omega}
\v\cdot\n\ d\H^{N-1}(\xx)\,\,,
\end{equation}
e.g. the classical linear elasticity formulation which achieves a finite minimum over $H^1(\om,\R^N)$ since the condition of equilibrated loads is fulfilled;
nevertheless with exactly the same choices there is a sequence $\w_h$ in $H^{1}(\om, \mathbf R^{N})$ such that $\F_h(\w_h)\rightarrow -\infty$
as $h\rightarrow 0^+$ (see Remark \ref{unbound}). Although minimizers of $\mathcal E$ over $H^1(\Omega;\R^N)$ exist, functionals $\F_h$ are not uniformly bounded from below.
These facts seem to suggest that, in presence of compressive forces acting on the boundary, the pure traction problem of linear elasticity cannot be deduced via $\Gamma$-convergence
from the nonlinear energy.\\
Moreover it is worth noting that if $\mathcal W$ fulfils \eqref{WquadIntrod} and $\mathbf g\equiv \mathbf f\equiv \mathbf 0\,$, hence $\inf \mathcal F_{h}= 0$ for every $h>0$,
then by choosing
a fixed nontrivial $N\times N$ skew-symmetric matrix $\mathbf W$, a real number $0< 2\alpha < 1$
and setting
\begin{equation}
\z_{h}:=h^{-\alpha}\,\mathbf W
\,\x \,,
\end{equation}
we get  $\ \mathcal F_{h}(\z_{h})= \inf \mathcal F_{h}+o(1)$, nevertheless
$\z_{h}$  has no subsequence
weakly converging in  $H^{1}(\om;\mathbf R^{N})$, see Remark \ref{rmk2.2}.
\\
Therefore here, in contrast to
\cite{DMPN},
we cannot expect weak $H^{1}(\om;\mathbf R^{N})$ compactness of minimizing sequences,
not even in the simplest case of null external forces: 
although this fact is common to pure traction problems in linear elasticity,
we emphasize that in nonlinear elasticity this difficulty cannot be easily circumvented in general
by standard translations since $\F_{h}(\v_{h})\!\not=\! \F_{h}(\v_{h}-\mathbb P\v_{h})$,
with $\mathbb P$ projection on infinitesimal rigid displacements.
\\
We deal this issue in the paper \cite{MPTJOTA}, showing nonetheless that at least for some special $\mathcal W$, if $ \mathcal F_{h}(\v_{h})= \inf \mathcal F_{h}+o(1)$ then up to subsequences $\mathcal F_{h}(\v_{h}-\mathbb P\v_{h})= \inf \mathcal F_{h}+o(1)$.
\vskip0.2cm
In order to have in general some kind of precompactness for sequences $\v_{h}$ fulfilling $ \mathcal F_{h}(\v_{h})=\inf \mathcal F_{h}+o(1)$, 
our approach consist in working
with a very weak notion, say weak $L^{2}(\om;\mathbf R^{N})$ convergence of linear strains: therefore the variational limit of $\mathcal F_{h}$ with respect to this convergence has to be investigated. Since  w-$L^{2}$ convergence of linear strains does not imply an analogous convergence of the skew symmetric part of the gradient of displacements, it can be expected that the $\Gamma\,$limit functional is different from the  point-wise limit of $\mathcal F_{h}$. \\
Indeed under some natural assumptions on $\mathcal W$, a careful application of the Rigidity Lemma of  \cite{FJM} shows that if $\mathbb E(\v_{h})$ are bounded in $L^{2}$ then, up to subsequences, $\sqrt h\nabla \v_{h}$ converges strongly in $L^{2}$ to a constant skew symmetric matrix and the variational limit of the sequence $\mathcal F_{h}$, with respect to the w-$L^{2}$ convergence of linear strains, turns out to be
a new functional
\begin{equation}
\label{DTfuncintro}
\F(\v)\ :=\ \displaystyle\min_{\mathbf W}\int_\Omega
\mathcal V_0\left( \,\xx,\,\mathbb E(\v) - \textstyle\frac{1}{2}\mathbf W^{2} \,\right)
\, d\xx\ -\ \mathcal L(\v) \,,
\end{equation}
where the minimum is evaluated over skew symmetric $N\!\times\! N$ matrices $\mathbf W$ and
\begin{equation}\label{QQ}
\mathcal V_0(\xx,\mathbf  B)\ :=\
\frac{1}{2} \,\mathbf  B^{T} D^2\V (\xx, \mathbf 0) \ \mathbf  B
\qquad \ \forall\,\mathbf B\in \M^{N\times N}_{sym}\,.
\end{equation}
We emphasize that the functional $\mathcal F$ in \eqref{DTfuncintro} is different
from the functional ${\mathcal  E}$ of linearized elasticity defined as
\begin{equation*}
{\mathcal  E}(\v):=\int_\Omega \mathcal V_0(\xx,\mathbb E(\v))\,d\x-\mathcal L(\v)\
\end{equation*}
since if $\v(\xx)= \textstyle\frac{1}{2} \mathbf W^{2}\xx$ with
 $\mathbf W\neq \mathbf{0}$ skew symmetric matrix,
 then $\F(\v)=-\mathcal L(\v)< {\mathcal  E}(\v)$.\\
 Nevertheless if $N=2$ then (see 
 Remark \ref{Fform}).
 $$
 \mathcal F(\vv) \ = \ \mathcal E(\vv)
-\frac{1}{4}\left (\int_{\om}\mathcal V_0(\xx,\mathbf I)
d\xx\right )^{\!-1}\left [\left (\int_{\om}D\mathcal V_0(\xx,\mathbf I)\!\cdot\!\mathbb E(\v)
\,d\xx\right)^{\!-}\right ]^{2} ,
 $$
hence $\F(\v)=\mathcal  E(\v)$ if
\begin{equation*}
N=2\,, \qquad \hbox{and}\qquad \int_{\om}D\mathcal V_{0}(x,\mathbf I)\cdot \mathbb E(\v)\,dx\ge 0\,.
\end{equation*}
In particular if $N=2$ and $\mathcal W$ is the {\it Green-St.Venant} energy density 
then the previous inequality reduces to
\begin{equation*}
\int_{\om} \dv\, \v\,d\xx \ge 0
\end{equation*}
which means, roughly speaking, that the area of $\om$ is less than the area of the related deformed configuration $\yy(\Omega)$, where $\yy(\xx)=\xx+h\vv(\xx)$ and $h>0$.
\\
The main results of this paper are stated in Theorems \ref{mainth1} and \ref{linel},
showing that under a suitable compatibility condition
on the forces (subsequent formula \eqref{compintro}) the pure traction problem in linear elasticity is deduced via $\Gamma$-convergence from
pure traction problem formulated in nonlinear elasticity,
referring to weak $L^2$ convergence of the linear strains. \\
Precisely Theorem \ref{mainth1} states that, if the loads $\mathbf f,\ \mathbf g$ fulfil \eqref{globalequiintro} together with the next compatibility condition
\beeq
\lab{compintro}
\int_{\partial\Omega}\!\!\mathbf f\cdot\mathbf W^{2}\x\, d\H^{N-1}
\!+\!
\int_{\Omega}\!\mathbf g\cdot\mathbf W^{2}\x\, d\xx\ <\ 0 \qquad\forall\,
\hbox{skew symmetric matrix } \mathbf W\!\not=\! \mathbf 0\,,
\eneq
then 
every sequence $\v_{h}$ with
$\F(\v_{h})=\inf\F_{h}+o(1)$ has a subsequence such that
the corresponding linear strains weakly convergence in $L^{2}$ to the linear strain of a minimizer of $\F$, together with convergence (without relabeling)
of energies $\F_{h}(\v_{h})$ to $\min \F$. 
Under the same assumptions Theorem \ref{linel} states that minimizers of $\F$ coincide with the ones of of linearized elasticity functional
$\mathcal E$, thus providing a full justification of  pure traction problems in linear elasticity at least if \eqref{compintro} is satisfied.
In particular, as it is shown in Remark 2.8,  this is true when $\mathbf g\equiv 0,\ \mathbf f=f\n$ with $f>0$ and $\n$ is the outer unit normal vector  to $\partial\om,$ that is when we are in presence of tension-like surface forces.
\\
Moreover, if there exists an $N\!\times\! N$ skew symmetric matrix such that the strict inequality
is reversed in \eqref{compintro}, then functional $\F$ is unbounded from below: see Remark \ref{controsegno} and Example \ref{nocrit}.
On the other hand if inequality in \eqref{compintro} is satisfied in a weak sense
by every skew symmetric matrix, then $\argmin \F$ contains $\argmin \mathcal E, \ \min \mathcal F=\min \mathcal E$  but $\F$
may have infinitely many minimizing critical points which are not minimizers of $\mathcal E$ (see Proposition \ref{exinfmin}).\\
Summarizing, only two cases are allowed: either $\min \mathcal F=\min \mathcal E$ or $\inf \F=-\infty$;
actually the second case arises in presence of compressive surface load.\\
By oversimplifying we could say
that $\F$ somehow preserves memory of instabilities which are typical of finite elasticity, while they disappear in the linearized model described by $\mathcal E$.\\
In the light of Theorem \ref{mainth1} and of remarks and examples of Section \ref{sectionlimpb}, it seems reasonable that,
as far as it concerns pure traction problems, the range of validity of linear elasticity should be restricted to a certain class
of external loads, explicitly those verifying \eqref{compintro}, a remarkable example in such class is a uniform normal tension load at the boundary as in Remark \eqref{tensatbdry};
while
in the other cases equilibria of a linearly elastic body could be better described through critical points of $\F$, whose existence in general seems to be an interesting and open problem.
\\
Concerning the structure of the new functional, we emphasize that actually $\mathcal F$ is different
from the classical linear elasticity energy functional $\mathcal E$, though there are many relations
between their minimizers (see Theorem \ref{linel}
). Further and more detailed information about functional $\mathcal F$ (a suitable property of weak lower semicontinuity, lack of subadditivity, convexity in 2D, nonconvexity in 3D) are described and proved in the paper \cite{MPTJOTA}.
\\\vskip0.2cm
\section{Notation and main result}
Assume that the reference configuration of an elastic body is a
\begin{equation}\label{OMEGA}
\hbox{ bounded, connected open set } \Omega \subset \R^N \hbox{ with Lipschitz boundary, }  \ N= 2, 3.\quad
\end{equation}
The generic point $\xx\in \Omega$ has components $x_j$ referring to the standard basis vectors
${\mathbf e}_j$ in $\R^N$;
${\mathcal L}^N$ and ${\mathcal B}^N$ denote respectively the
$\sigma\mbox{-algebras}$
of Lebesgue measurable and Borel measurable subsets of $\R^N$. For every $\alpha\in \mathbb R$ we set $\alpha^+=\alpha\vee 0,\ \alpha^-=-\alpha\vee 0$.
\vskip0.1cm
The notation for vectors $\aa,\,\bb\in\R^N$ and $N \!\times \!N$ real matrices $\mathbf A,\,\mathbf B,\, \mathbf F $ are as
follows: $\aa\cdot\bb=\sum_j\aa_j\bb_j\,;$
$\mathbf A\cdot\mathbf B=\sum_{i,j}\mathbf A_{i,j}\mathbf B_{i,j}\,;$
$[\mathbf A\mathbf B]_{i,j}=\sum_{k}\mathbf A_{i,k}\mathbf B_{k,j}\,;$
$| \mathbf F|^2 =\hbox{Tr}(\mathbf F^{T}\mathbf F)=\sum_{i,j}F_{i,j}^2$ denotes the squared Euclidean norm of $\mathbf{F}$ in the space $ \M^{N \times N}$ of $N\!\times\! N$ real matrices;
$\mathbf I \in \M^{N \times N}$ denotes the identity matrix,
$SO(N)$ denotes the group of rotation matrices, $\M^{N\times N}_{sym}$ and $\M^{N\times N}_{skew}$ denote respectively the sets of symmetric and skew-symmetric matrices.
For every $\mathbf B\in \M^{N \times N}$ we define ${\rm sym\,}\mathbf B:=\frac{1}{2}(\mathbf B+\mathbf B^T)$
and  ${\rm skew\,}\mathbf B:=\frac{1}{2}(\mathbf B-\mathbf B^T)$.
\vskip0.1cm
It is well known that the matrix exponential maps $\M^{N\times N}_{skew}$ to $SO(N)$ and is surjective on $SO(N)$ (see \cite{H}).
Therefore for every $\mathbf R\in SO(N)$ there exist $\vartheta\in \mathbb R$ and $\mathbf W\in \M^{N\times N}_{skew},$ $|\mathbf W|^{2}=2$ such that $\exp(\vartheta\,\mathbf W)= \mathbf R$. By taking into account that $\mathbf W^{3}=-\mathbf W$, the Taylor's series expansion of $\,\vartheta\to \exp(\vartheta\,\mathbf W)=\sum_{k=0}^\infty \vartheta^k\mathbf W^k/k!\,$ yields the {\it Euler-Rodrigues formula}:
\begin{equation}\label{eurod}
\exp(\vartheta\,\mathbf W)\,=\,\mathbf R\,=\,\mathbf I\,+\,\sin\vartheta \,\mathbf W\,+\,(1-\cos\vartheta)\,\mathbf W^{2}\,.
\end{equation}
In particular from \eqref{eurod}  it follows that if we set
\begin{equation}\label{cone}
\mathbb K:=\{\tau(\mathbf R-\mathbf I): \tau> 0,\  \mathbf R\in SO(N)\}
\end{equation}
then we obtain
\begin{equation}\label{clcone}
\overline{\mathbb K}={\mathbb K}\cup {M^{N\times N}_{skew}}.
\end{equation}
\\
For every $\mathcal U:\om\times \M^{N \times N}\rightarrow \mathbb R,$ with $ \mathcal U(\xx,\cdot)\in C^{2}$ a.e. $\xx\in \om$,
we denote
by $D\mathcal U(\xx,\cdot)$ and $D^{2}\mathcal U(\xx,\cdot)$ respectively
the gradient and the hessian of $g$ with respect to the second variable.
\\
For every displacements field $\vv\in H^1(\Omega;\R^N)$,  $\E(\vv)\!:=\! \hbox{sym }   \!\nabla\vv$ denotes the infinitesimal strain tensor field,
$\mathcal R\!:=\!\{ \vv\in H^1(\Omega;\R^N): \E(\v)=\mathbf 0\}$ denotes the space spanned by set of the infinitesimal rigid displacements and
$\mathbb P\vv$ is the orthogonal projection of $\vv$ onto $\mathcal R$.\\
We set $-\!\!\!\!\!\int_\om \vv d\xx =|\om|^ {-1 } \int_\om\vv d\xx.$
\\
We consider a body made of an hyperelastic material, say
there exists a ${\mathcal L}^N\! \!\times\! {\mathcal B}^{N^2} $measurable
$\W : \om \times \M^{N \times N} \to [0, +\infty ]$
such that,  for a.e. $\xx \in \om$,
$\W(\xx,\nabla \mathbf y(\xx))$ represents the stored energy density, when $\mathbf y(x)$ is the deformation and $\nabla \mathbf y(\xx)$ is the deformation gradient.\\
Moreover we assume that for a.e. $\xx \in \om$
\beeq\lab{incom}  
	\W(\xx,\mathbf F)=+\infty \qquad \mbox{if $\det \mathbf F \leq 0$} \quad\hbox{(orientation preserving condition)}\,,
\eneq
\beeq \lab{framind} \W(\xx, \mathbf R\mathbf F)=\W(\xx, \mathbf F)\qquad \forall \, \mathbf R\!\in\! SO(N) \quad
\forall\, \mathbf F\in \M^{N \times N}\ \quad  \hbox{(frame indifference)}\,,
\eneq
\beeq\lab{reg} \exists\   \hbox{a neighborhood} \ \mathcal A  \  \hbox{of}\ SO(N) \hbox{ s.t.}
\quad\mathcal W(\xx,\cdot)\in C^{2}(\mathcal A)\,,
\eneq
\beeq \lab{coerc}
\exists\, C\!>\!0 \hbox{ independent of }\xx:\ 
\W(\xx,\mathbf F)\ge  C|\mathbf F^{T}\mathbf F-\mathbf I|^{2}\ \,\forall\, \mathbf F\!\in\! \M^{N \times N}\,   \hbox{(coerciveness)},
\eneq
\beeq \lab{Z1}
	\W(\xx,\mathbf I)=0\,, \qquad D \W(\xx,\mathbf I)=0\,, \qquad \hbox{for a.e. }\xx \in \om  \,,
\eneq
that is  the reference configuration
has zero energy and
is stress free, so
by \eqref{framind}  we get also
$$\W(\xx,\mathbf R)\!=\!0,\ D \W(\xx,\mathbf R)\!=\!0 \qquad \forall \,\mathbf R\in SO(N)\,.$$
By frame indifference there exists a  ${\mathcal L}^N \!\times\! {\mathcal B}^N \mbox{-measurable}$ $\V : \om \times \M^{N \times N} \to [0,+\infty]$ such that for every $\mathbf F\in \M^{N\times N}$
\beeq
\lab{WhatZ}
\W(\xx,\mathbf F)=\V(\xx,\textstyle{\frac{1}{2}}( \mathbf F^T \mathbf F - \mathbf I))
\eneq
and by \eqref{reg}
%
\beeq\lab{regV} \exists\   \hbox{a neighborhood} \ \mathcal O \  \hbox{of}\ \mathbf 0 \hbox{ such that}
\ \mathcal V(\xx,\cdot)\in C^{2}(\mathcal O), \hbox{ a.e. }x \in \om \,.
\eneq
In addition we assume that  there exists $\gamma>0$ independent of $\xx$ such that
\beeq \lab{Velliptic}
	\left|\,\mathbf B^T\,D^2 \mathcal V(\xx,\mathbf D) \,\mathbf B\,\right| \, \leq \, 2\, \gamma\, |\mathbf B|^2
	\quad
	\forall\, \mathbf D\!\in\! \mathcal O, \ \forall\, \mathbf B \!\in\! \M^{N \times N} .
\eneq


By (\ref{Z1}) and Taylor expansion with Lagrange reminder we get,
for a.e. $\xx \in \om$ and suitable $t \in (0,1)$ depending on $\xx$ and on $\mathbf B$:
\beeq \lab{12'}
	\mathcal V(\xx,\mathbf  B)\ =\ \frac{1}{2}\, \mathbf  B^{T} D^2\V (\xx, t \mathbf  B)\, \mathbf  B\,.
\eneq
Hence  by \eqref{Velliptic}
\beeq \lab{upbound}
	\mathcal V(\xx,\mathbf  B) \, \leq \, \gamma \,|\mathbf  B|^2
	\qquad \forall \ \mathbf  B\in \M^{N \times N}\cap  \mathcal O\,.
\eneq
According to \eqref{WhatZ}
for a.e.  $ \xx \!\in\! \om,\ h\!>\!0$ and every $ \mathbf  B\in \M^{N \times N} $ we set
\beeq \lab{Wh}
	\mathcal V_{h} (\xx,\mathbf B) := \frac{1}{h^2} \, \mathcal W(\xx,\Id+h \mathbf B) =
	\frac{1}{h^2} \, \mathcal V(\xx,h \, \hbox{\rm sym\,} \mathbf B +
	\textstyle{\frac{1}{2}}h^2 \mathbf B^T \mathbf B)\ .
\eneq
Taylor's formula with \eqref{Z1},\eqref{Wh} entails
$\mathcal V_h(\xx, \mathbf B)= \frac 1 2 \, (\hbox{\rm sym\,} \mathbf B)\,
D^2\mathcal V(\xx,\mathbf 0) \,
 (\hbox{\rm sym\,} \mathbf B) + o(1)$, so
\begin{equation}\lab{Vh}
 \mathcal V_h(\xx, \mathbf B)\ \to\ \mathcal V_0 (\xx, \hbox{\rm sym } \mathbf B) \hbox{ as } h\to 0_+\,,
 \end{equation}
where the point-wise limit of integrands is the quadratic form $\mathcal V_0$ defined by
 \begin{equation}\label{A e Q}
    \mathcal V_0 (\xx,\mathbf B)
    \ :=\ \frac 1 2 \,\mathbf B^TD^2\mathcal V(\xx,\mathbf 0) 
  \,\mathbf B\,
  \quad \hbox{a.e. }\,\xx\in\Omega,\ \mathbf B\in \M^{N\times N}\,.
  \end{equation}
The symmetric fourth order tensor $D^2\mathcal V(\xx,\mathbf 0)$ in \eqref{A e Q} plays the role of the classical elasticity tensor.
\\
By 
\eqref{coerc} we get
\begin{equation}\label{(2.18)}
  \mathcal V_{h} (x,\mathbf B)\ =\ \frac{1}{h^2} \, \mathcal W(x,\Id+h \mathbf B)\ \ge\
  C\,|\,2\,{\rm sym } \mathbf B\,+\,h\,\mathbf B^T \mathbf B\,|^{2}
\end{equation}
so that \eqref{A e Q} and \eqref{(2.18)} imply the ellipticity of $\mathcal V_0$\,:
\beeq \lab{Z3}
	\mathcal V_0 (\xx,\hbox{\rm sym}\,\mathbf B)
	\, \geq \, 4\,C \, |\hbox{\rm sym}\,\mathbf  B|^2
	\qquad 
\ \hbox{a.e. }\,\xx\in\Omega,\ \mathbf B\in \M^{N\times N}\,.
\eneq
%
Let $\mathbf f\in L^2(\partial \Omega;\R^N),\ \mathbf g\in L^{2}(\Omega;\R^N)$ be a pair of surface and body force fields respectively.\\
For a suitable choice of the adimensional parameter $h>0$, the functional representing the total energy is labeled by $\mathcal F_h: H^1(\Omega;\R^N)\to \R\cup\{+\infty\} $ and defined as follows
\begin{equation}
\label{nonlinear}
\displaystyle \mathcal F_h(\v)\ :=\
\int_\Omega \mathcal V_{h}(\xx, \nabla\v)\, d\xx\,-\,\mathcal L(\v)\,,
\end{equation}
where
\beeq
\lab{potential} \mathcal L(\v):\ =\ \int_{\partial \Omega}\mathbf f\cdot\v
\ d\H^{n-1}\, +\,\int_{\Omega} \mathbf g\cdot\mathbf v\ d\xx.
\eneq
%
In this paper we are interested in the  asymptotic behavior as $h\to 0_+$
of functionals $\mathcal F_h$
and to this aim we
introduce the limit energy functional $\mathcal F: H^1(\Omega;\R^N)\to \R$ defined by
\begin{equation}
\label{DTfunc}
\F(\v)\ =\ \displaystyle\min_{\mathbf W\in \M^{N\times N}_{skew}}\int_\Omega 
\mathcal V_0\left(\xx,\mathbb E(\v)-\textstyle\frac{1}{2}\mathbf W^{2}\right)
 d\xx \,- \,\mathcal L(\v)\, .
\end{equation}
We emphasize that the minimum in right-hand side of definition \eqref{DTfunc} exists:  precisely the finite dimensional minimization problem has exactly two solutions which differs only by a sign, since by
\eqref{Z3},
\begin{equation}
\displaystyle\lim_{|\mathbf W| \to +\infty,\, \mathbf W\in \M^{N\times N}_{skew}}\int_\Omega
{\mathcal V_0}\left(\xx,\mathbb E(\v)-\textstyle\frac{1}{2}\mathbf W^{2}\right)
\, d\xx\ =\ +\infty\
\end{equation}
and ${\mathcal V_0}(\xx,\cdot)$ is strictly convex by \eqref{A e Q},\,\eqref{Z3}.
\vskip0.3cm
All along this paper we assume
\eqref{OMEGA} together with the \textsl{standard structural conditions} \eqref{incom}-\eqref{Z1},\eqref{Velliptic} 
as usual in scientific literature concerning elasticity theory
and we refer to the notations
\eqref{WhatZ},\eqref{Wh},\eqref{A e Q},\eqref{nonlinear}-\eqref{DTfunc}.

\vskip0.3cm
The pair $\mathbf f,\,\mathbf g$ describing the loads is said to be {\it equilibrated} if
\begin{equation}
\label{globalequi}
\int_{\partial\Omega} \mathbf f\cdot\mathbf z\,d\H^{N-1}+\int_{\Omega} \mathbf g\cdot\mathbf z\,d\xx\ =\ 0 \quad \forall\;
\mathbf z\in \mathcal R \,,
\end{equation}
and it is said to be {\it compatible} if
\beeq
\lab{comp}
\int_{\partial\Omega}\mathbf f\cdot\mathbf W^{2}\,\x\ d\H^{N-1}\,+\,\int_{\Omega}\mathbf g\cdot\mathbf W^{2}\,\x\ d\xx \ <\ 0\qquad\ \forall\ \mathbf W\in \M^{N\times N}_{skew}\ \ \hbox{s.t.}\ \mathbf W\not= \mathbf 0.
\eneq
\begin{definition}\label{minimseq}
Given an infinitesimal sequence $h_j$  of positive real numbers, we say that $\v_j\in H^1(\Omega;\R^N)$ is a \textsl{minimizing sequence} of the sequence of functionals $\mathcal F_{h_{j}}$
if $$(\mathcal F_{h_{j}}(\vv_j)-\inf\mathcal F_{h_j})\to 0 \quad \hbox{ as } \quad h_j\to 0_+ \,.$$
\end{definition}
We will show (see Lemma \ref{infFh}) that for every infinitesimal
sequence $h_j$ the minimizing sequences of the sequence of functionals $\mathcal F_{h_j}$
do exist.
\\
Now we can state the main result, whose proof is postponed.
\begin{theorem}
\label{mainth1}
Assume that
the \textsl{standard structural conditions} and
\eqref{globalequi},\eqref{comp} hold true.
Then for every sequence of strictly positive real numbers $h_{j}\to 0$
there are
\textsl{minimizing sequences} of the sequence of functionals $\mathcal F_{h_{j}}$.\\
Moreover for every minimizing sequence $\v_j\in H^1(\Omega;\R^N)$
of $\mathcal F_{h_{j}}$ there exist:
a subsequence, a displacement
$\v_0\in H^1(\Omega;\R^N)$ and
a constant matrix
$\mathbf W_0\in \M^{N\times N}_{skew}\,$
 such that, without relabeling,
\begin{equation}
\label{wstr}
\mathbb E(\v_j)\ \wconv\ \mathbb E(\v_{0})\qquad
 weakly \ in \ L^2(\Omega;\M^{N\times N})\,,\quad
\end{equation}
\begin{equation}\label{convsqrth}
\sqrt{h_{j}}\ \nabla\v_j\ \rightarrow\ \mathbf W_0 \qquad strongly\ in \ L^2(\Omega;\M^{N \times N})\,,
\end{equation}
\begin{equation}
\label{min}
\lim_{j\to +\infty}\mathcal F_{h_{j}}(\v_j)\ \ =\ \mathcal F(\v_0)\ \ =\ \
\min_{\v\in H^1(\Omega;\R^N)} \! \mathcal F(\v)\,,\qquad
\end{equation}
\begin{equation}\label{effe0}
\displaystyle \mathcal F(\v_0)=\int_{\om}\mathcal V_{0}(\xx, \mathbb E(\v_{0})-\textstyle\frac{1}{2}\mathbf W_{0}^{2})\,d\xx - \mathcal L(\v_{0})\ .\qquad
\end{equation}
\end{theorem}
\begin{remark}\label{rmk2.2}
{\rm It is worth to underline that in contrast to the case of Dirichlet problem faced in \cite{DMPN},
here in pure traction problem we cannot expect even weak $H^{1}(\om;\mathbf R^{N})$ convergence of minimizing sequences.\\
Indeed choose: $\mathbf f=\mathbf g\equiv 0$ and
\begin{equation}
\lab{Wquad}
\W(\xx,\mathbf F)=\left\{\begin{array}{ll}
 &|\mathbf F^{T}\mathbf F-\Id|^{2}\  \hbox{if} \ \det \mathbf F>0\,,
 \vspace{0.1cm}\\
  &+\infty\ \hbox{ otherwise\,,}
 \end{array}\right.
\end{equation}
\begin{equation}
\v_j:=h_{j}^{-\alpha}\,\mathbf W
\,\x\qquad \hbox{with }\  \mathbf W\!\in\! \M^{N\times N}_{skew}\,, \
0< 2\alpha < 1\,,\ h_{j}\to 0_{+}\,.
\end{equation}
Then $\ \mathcal F_{h_{j}}(\v_j)= o(1)$ and, due to
$\inf \mathcal F_{h_{j}}= 0$, the sequence $\v_j$ is a minimizing sequence which has no subsequence
weakly converging in  $H^{1}(\om;\mathbf R^{N})$.
It is well known that such phenomenon takes place for pure traction problems
in linear elasticity too, but in nonlinear elasticity this difficulty cannot be easily circumvented in general,
since the fact that $\v_{j}$ is a minimizing sequence does not entail that also $\v_{j}-\mathbb P\v_{j}$ is
minimizing sequence.
In \cite{MPTJOTA} we show that for some special integrand $\mathcal W$, as in the case of {\it Green-St-Venant energy density}, if $\v_{j}$ is a minimizing sequence  then  $\w_j:=\v_{j}-\mathbb P\v_{j}$ is a minimizing sequence too and there exist a (not relabeled) subsequence of functionals  $\F_{h_{j}}$ such that the related {\it minimizing subsequence} $\w_{j}$ converges weakly in $H^{1}(\om; \mathbb R^{N})$ to a minimizer $\v_{0}$ of $\F$, provided \eqref{globalequi} and \eqref{comp} hold true.}
\end{remark}
\begin{remark}
{\rm A careful inspection of the proof  
shows that Theorem~\ref{mainth1} remains true if
  hypothesis \eqref{coerc} is weakened by assuming
\beeq\lab{V1}
\displaystyle \inf_{|\mathbf B|\ge\rho}\, \inf_{x\in\om} \,\V(x,\mathbf B)>0\ \
\forall\ \rho>0,
\eneq
\beeq\lab{V2}
\exists\, \alpha>0,\  \rho>0\ \ \hbox{such that}
\displaystyle \inf_{x\in\om} \V(x,\mathbf B)\ge\alpha  |\mathbf B|^2\ \ \forall\ \ |\mathbf B|\le\rho
\eneq
and
\beeq\lab{V3}
\displaystyle \liminf_{|\mathbf B|\to +\infty}\,  \frac{1}{|\mathbf B|} \,
\inf_{x\in\om} \V(x,\mathbf B)>0.
\eneq
This can be shown by exploiting Lemma 3.1 in \cite{DMPN}; notice that under such weaker assumption the constant appearing on the right-hand side of inequality \eqref{infFhdisplay} must be modified accordingly in Lemma \ref{infFh}.
\\
It is worth noting that  \eqref{coerc} implies  \eqref{V1},  \eqref{V2} and  \eqref{V3}.}
\end{remark}
Here we show some preliminary remarks about the main result
and the limit functional $\mathcal F$.
\begin{remark}
\lab{Fform}
\rm If $N=2$,
then for every $W\in \M^{N\times N}_{skew}$ there is $a\in\mathbb R$ s.t. $\mathbf W^2=- a^2\mathbf I$, hence
\eqref{DTfunc}  reads
\begin{equation}\lab{dta}
\F(\v)\,\,=\displaystyle\min_{a\in \mathbf R}\int_\Omega \mathcal V_{0}(\xx, \mathbb E(\v)+\textstyle\frac{a^{2}}{2}\mathbf I)\, d\xx- \mathcal L(\v)\,,
\end{equation}
therefore, a minimizer $a_{*}(\v)$ of functional \eqref{dta} (with respect to $a\in\R$ with fixed $\vv$) 
fulfils
\begin{equation*}
a_{*}^{3}(\v)\int_{\om}\mathcal V_{0}(x,\mathbf I)\,d\xx+a_{*}(\v)\int_{\om}D\mathcal V_{0}(\xx,\mathbf I)\cdot \mathbb E(\v)\,d\xx=0
\end{equation*}
that is
\begin{equation}\label{astar}
a_{*}^{2}(\v)=\left (\int_{\om}\mathcal V_{0}(\xx,\mathbf I)\,dx\right )^{-1}\left (\int_{\om}D\mathcal V_{0}(\xx,\mathbf I)\cdot \mathbb E(\v)\,d\xx\right)^{-}
\end{equation}
and
\begin{equation}
\displaystyle \mathcal F(\v)\ =\ \int_{\om}\mathcal V_0\left (\xx\,,\,\mathbb E(\v)+\textstyle\frac{a_{*}^{2}(\v)}{2}\mathbf I\right)\,d\xx\,-\,\mathcal L(\v)\ .
\end{equation}
By taking into account that $\mathcal V_{0}$ is a quadratic form,
we can make explicit the gap between the new functional $\mathcal F$ and the classical linear elasticity functional $\mathcal E$, when $N=2$\,:
\begin{equation}\lab{F}
\begin{array}{rl}
&\displaystyle \mathcal F(\v)\ = \\
& =\displaystyle\int_{\om}\mathcal V_0\left (\xx,\mathbb E(\v)\right)
\,d\xx -\frac{1}{4}\left (\int_{\om}\mathcal V_0(\xx,\mathbf I)
d\xx\right )^{\!-1}\left [\left (\int_{\om}D\mathcal V_0(\xx,\mathbf I)\!\cdot\!\mathbb E(\v)
\,d\xx\right)^{\!-}\right ]^{2}\!\!\!-\mathcal L(\v) \\
&=\displaystyle
\mathcal E(\vv)
-\frac{1}{4}\left (\int_{\om}\mathcal V_0(\xx,\mathbf I)
d\xx\right )^{\!-1}\left [\left (\int_{\om}D\mathcal V_0(\xx,\mathbf I)\!\cdot\!\mathbb E(\v)
\,d\xx\right)^{\!-}\right ]^{2}\,.
\end{array}
\end{equation}
In particular, if $N=2$, $\,\lambda,\, \mu> 0\,$ and
\begin{equation}\lab{GSV1}
\W(\xx,\mathbf F)=\left\{\begin{array}{ll}
 &\mu |\mathbf F^{T}\mathbf F-I|^{2}+\frac{\lambda}{2}|\ \hbox{\rm Tr}\ (\mathbf F^{T}\mathbf F-I)|^{2}
 \quad  \hbox{if} \ \det \mathbf F>0
 \vspace{0.2cm}\\
 &+\infty\quad \hbox{ otherwise,}\\
 \end{array}\right.
\end{equation}
then $\mathcal V_0(\xx,\mathbf B)\,=\,4\mu|\mathbf B|^{2}+2\lambda|\hbox{Tr}\mathbf B|^{2}$ and we get
\begin{equation}\label{astarGSV}
a_{*}^{2}(\v)\ =\ |{\om}|^{-1}\left (\int_{\om}\hbox{\rm div}\,\v\,d\xx\right)^{\!-}\ .
\end{equation}
Roughly speaking, this means that in 2D the global energy $\mathcal F(\v)$ of a displacement $\v$ is the same of linearized elasticity
if the area of the associated deformed configuration $
\yy
(\Omega)$ is greater than
the area of $\om$.
\end{remark}
\begin{remark}\lab{unbound}
{\rm The compatibility condition \eqref{comp} cannot be dropped in Theorem \ref{mainth1}
even if the (necessary) condition \eqref{globalequi} holds true. Moreover plain substitution of strong with weak inequality
in \eqref{comp} leads to a lack of compactness for minimizing sequences.
\vskip0.1cm
Indeed, if $\n$ denotes the outer unit normal vector to $\partial\om$ and we choose
$\mathbf f=f\n$ with $f<0$, $\mathbf g\equiv 0$ then
\beeq
\int_{\partial\Omega}\mathbf f\cdot\mathbf W^{2}\,\x\,d\H^{N-1}\ =\  2f(\Tr\mathbf W^{2})\vert\Omega\vert \ >\ 0
\qquad \forall\ \mathbf W\!\in\! \M^{N\times N}_{skew}\setminus \{\mathbf 0\}\,,
\eneq
say, the strict inequality in \eqref{comp} is reversed in a strong sense by any $\mathbf W\!\in\! \M^{N\times N}_{skew}\setminus \{\mathbf 0\}$;\\
fix a sequence of positive real numbers such that $h_{j}\!\to\! 0$\,, 
$\mathbf W\!\in\! \M^{N\times N}_{skew},\  \mathbf W\not\equiv \mathbf 0$,
and set $\v_j={h_{j}}^{-1}(\frac{1}{2}\mathbf W^2+\frac{\sqrt 3}{2}\mathbf W)\,\xx\ $;
then $\mathbf I +\big(\frac 1 2 \mathbf W^2+\frac {\sqrt 3}{2}\mathbf W\big)\in SO(N) $ and, by frame indifference,
\begin{equation}
\mathcal F_{h_{j}}(\v_j)=-\frac{f}{2h_{j}}\int_{\partial \Omega}\mathbf W^{2}\x\cdot\n\,d\H^{n-1}=-\frac{f}{2h_{j}}(\Tr\mathbf W^{2})|\Omega|\rightarrow -\infty.
\end{equation}
On the other hand, assume
$\W$ as in \eqref{Wquad} and
$\mathbf f\!=\!\mathbf g\!\equiv\! \mathbf 0$, so that the compatibility inequality is susbstituted by the weak inequality;
%
if $\v_j$ are defined as above then,
hence
 \begin{equation}
\F_{h_{j}}(\v_j)=0=\inf \F_{h_{j}}
\end{equation}
but $\mathbb E(v_j)$ has no weakly convergent subsequences in $L^2(\Omega;\M^{N\times N})$.\\
}
\end{remark}
%
\begin{remark} \label{tensatbdry}\rm It is worth noticing that the compatibility condition \eqref{comp} holds true
when $\mathbf g\equiv 0$, $\mathbf f=f\n$ with $f>0$ and $\n$ the outer unit normal vector to $\partial\om$.\\
Indeed let $\mathbf W\in \M^{N\times N}_{skew}, \mathbf W\not\equiv \mathbf 0$:
hence by  \eqref{globalequi} and the Divergence Theorem we get
\beeq
\int_{\partial\Omega}\mathbf f\cdot\mathbf W^{2}\,\x\,d\H^{N-1}= 2f(\Tr\mathbf W^{2})\vert\Omega\vert < 0
\eneq
thus proving \eqref{comp} in this case.
%
Roughly speaking this means that in presence of tension-like surface forces and of null body forces the compatibility condition holds true.\\
\end{remark}
\begin{remark}
{\rm  It is possible to  observe some analogy between the energy functional \eqref{DTfunc} and the results in
\cite{DDT},\cite{DDT2} where the approximate theory of small strain accompanied
by \textsl{moderate rotations} is discussed under suitable  kinematical assumptions.
More precisely, if $\mathbf F= \mathbf I+ h\nabla\v$ is the deformation gradient, $\mathbf F=\mathbf R\mathbf U$ the polar decomposition, \cite{DDT} shows that the assumptions
\begin{equation}\lab{assrot}
\mathbf R=\mathbf I+O(\sqrt h),\ \ \mathbf U=\mathbf I+O(h)
\end{equation} as $h\rightarrow 0$ (in the sense of pointwise convergence) are equivalent to
\begin{equation}
\mathbb E(\v)=O(1),\ \ h ({\rm skew}\nabla\v)=O(\sqrt h)
\end{equation}
 as $h\rightarrow 0$ again in the sense of pointwise convergence. Therefore
\begin{equation}\lab{sviluppo}
\mathbf U= \mathbf I+h\big ( \mathbb E(\v)-\textstyle\frac{1}{2}({\rm skew}\nabla\v\big)^{2}) + o(h)
\end{equation}
 and  the point-wise limit  of $\F_{h}$ ( not the $\Gamma$-limit !) becomes
 \begin{equation*}
\int_\Omega \mathcal V_{0}\big(\xx, \mathbb E(\v)-\textstyle\frac{1}{2}({\rm skew}\nabla\v)^{2}\big)\, d\xx- \mathcal L(\v) \,,
    \end{equation*}
which is quite similar to \eqref{DTfunc}.
    \\
We highlight the fact that \eqref{assrot} cannot be understood in the sense of $L^{2}(\om,\M^{N\times N})$
whenever $\v\equiv \v_{*}$ on a closed subset $\Sigma$ of
$\partial \om$ with ${\mathcal H}^{n-1} (\Sigma)>0$, since by Korn and Poincar\`e inequalities we get
\begin{equation*}
\int_{\om}|\nabla \v|^{2}\,d\xx\le C\Big (\int_{\om}|\mathbb E(\v)|^{2}\,d\xx+\int_{\Sigma}|\v_{*}|^{2}\,d\H^{N-1}\Big )\,,
\end{equation*}
therefore if $\mathbb E(\v)=O(1)$ then $h\nabla \v= O(h)$, thus contradicting the second of \eqref{assrot}.
On the other hand   a careful application of  the rigidity Lemma of \cite{FJM} show that if $\mathbb E(\v)=O(1)$ and $ \mathbf U= \mathbf I+ O(h)$ in the sense of $L^{2}(\om,\M^{N\times N})$, then there exists  a constant skew symmetric matrix $\mathbf W$ such that $h\nabla\v^{T}\nabla\v =-\mathbf W^{2}+o(1)$ in the sense of $L^{1}(\om,\M^{N\times N})$ (see the proof of Lemma \ref{convfunc} below). Therefore
\begin{equation}\lab{sviluppobis}
\mathbf U= \mathbf I+h\big( \mathbb E(\v)- {\mathbf W^{2}}/{2}\big) + o(h)
\end{equation}
where equality is understood in the sense of $L^{1}(\om,\M^{N\times N})$ and $\mathbf W$ a constant skew symmetric matrix.}
\end{remark}
%
\section{Proofs}\label{Section Proofs}
We recall three basic inequalities exploited in the sequel, for reader's convenience and in order to label the related constants.
\vskip0.1cm
{\bf Poincar\`e Inequality}. There exists a constant $C_{P}=C_P(\om)$ such that
\begin{equation}\label{poi}
\left\|\,\v - \hbox{$- \!\!\!\!\!\int_{\,\Omega}$}\, \v \,\right\|_{L^2(\Omega;\R^N)} \ \leq \ C_P\,\|\nabla w\|_{L^2(\Omega;\M^{N \times N})} \qquad \forall \,\v\in H^1(\om;\R^N)\ .
\end{equation}
\vskip0.1cm
{\bf Korn-Poincar\`e Inequality}. There exists a constant $C_{K}=C_K(\om)$ such that
\begin{equation}\label{kornpoi}
\|\vv-\mathbb P\vv\|_{L^2(\Omega;\R^N)}\,+\,
\|\vv-\mathbb P\vv\|_{L^2(\partial\Omega;\R^N)}
\,\le\ C_{K}\ \|\mathbb E(\vv)\|_{L^2(\Omega;\M^{N \times N})}
\quad \forall \,\v\in H^1(\om;\R^N)\ .
\end{equation}
\vskip0.1cm
{\bf Geometric Rigidity Inequality} (\cite{FJM}). There exists a constant $C_G=C_{G}(\om)$ such that for every $\yy\in H^{1}(\om;\mathbb R^{N})$
there is an associated rotation $\mathbf R\in SO(N)$ such that we have
\begin{equation}\label{muller}
\int_{\om}|\nabla\yy-\mathbf R|^{2}\,d\xx\le C_{G}\int_{\om}\hbox{dist}^{2}(\nabla \yy; SO(N))\,d\xx.
\end{equation}
The first step in our analysis is the next lemma showing that if \eqref{globalequi},\,\eqref{comp} hold true then the functionals $\F_{h}$  are bounded
from below uniformly with respect to $h\in\N$: this implies the existence of minimizing sequences of the sequence of functionals $\F_{h_j}$ (see Definition \ref{minimseq}).
\begin{lemma}\lab{infFh}  Assume \eqref{globalequi} and \eqref{comp}. Then
\begin{equation}\label{infFhdisplay}
  \displaystyle\inf_{h>0}\,\inf_{\vv\in H^1} \ \F_{h}(\vv)\ > \ -\, \frac {C_P^{2}\,C_G}{C}\,\big(\|\mathbf f\|_{L^2}^2+\|\mathbf g\|_{L^2}^2\big)\,,
\end{equation}
where $C$ is the coercivity constant in \eqref{(2.18)} and $C_P,
\,C_G$ are the constants related to the basic inequalities above. %
\\
Actually \eqref{infFhdisplay} holds true even if strict inequality is replaced by weak inequality in \eqref{comp}.
\end{lemma}
\begin{proof} 
Let $\v\in H^1(\Omega;\R^N)$ and $\yy=\x+h\v$.
Since $\det \nabla \yy>0\, $a.e., by polar decomposition for a.e. $\xx$ there exist a rotation $\mathbf R_h(\xx)$ and a symmetric positive definite matrix $\mathbf U_h(\xx)$ such that $\nabla \yy(\xx)=\mathbf R_h (\xx)\mathbf U_h(\xx)$, hence $\nabla\yy^{T}\nabla \yy=\mathbf U_h^2$, so that for a.e. $\xx$
\begin{equation}\label{polardec}
\begin{array}{lll}
  |\nabla\yy^{T}\nabla \yy-\mathbf I|^{2}
  &=&
  |\mathbf U_h^2-\mathbf I|^{2}=
  |(\mathbf U_h-\mathbf I)(\mathbf U_h+\mathbf I)|^{2} \geq
  |\mathbf U_h-\mathbf I|^{2}= \vspace{0.2cm}\\
   &=&
  |(\nabla \yy-\mathbf R_h)|^{2}\,\geq \, {\rm dist} ^2\!\left( \nabla \yy, SO(N)\right) .
  \end{array}
\end{equation}
By \eqref{coerc},\eqref{polardec} and the Geometric Rigidity Inequality \eqref{muller} there exists a constant rotation $\mathbf R$ such that
\begin{equation}\label{lowerbd}\begin{array}{ll}
&\displaystyle\F_{h}(\v)\ge Ch^{-2}
\int_{\Omega}|\nabla\yy^{T}\nabla \yy-\mathbf I|^{2}\,d\x-h^{-1}\mathcal L(\yy-\x)\ge\\
&\\
&\displaystyle\ge \frac C {C_G}\,h^{-2}\,\int_{\Omega}|\nabla \yy-\mathbf R|^{2}\,d\x-h^{-1}\mathcal L(\yy-\x).\\
\end{array}
\end{equation}
If now
$$\mathbf c:=|\Omega|^{-1}\int_{\Omega}(\yy-\mathbf R\x)\,d\x$$
by \eqref{globalequi} , by Poincar\`e and Young
 inequality we get, for every $\alpha > 0$,
 \begin{equation*}
   \|\yy-\mathbf R\xx-\mathbf c \|_{L^2} \ \leq \ C_P\,\| \nabla (\yy - \mathbf R\xx ) \|_{L^2} \ =
    \ C_P\,\| \nabla \yy - \mathbf R \|_{L^2}\ ,
 \end{equation*}
 \vskip0.1cm
 \begin{equation*}\begin{array}{ll}
&\mathcal L(\yy-\mathbf R\x-\mathbf c)\,\le\, C_P\,\|\nabla \yy - \mathbf R\|_{L^2}\, \big(\|\mathbf f\|_{L^{2}}+\|\mathbf g\|_{L^{2}}\big)
\ \leq
\\
&\\
& \displaystyle
\leq \ \alpha^{-1}\, \frac {C_P} {2}\, \|\nabla \yy - \mathbf R\|_{L^2}^2 \ +\ \alpha\, \frac {C_P} 2 \,
\big(\|\mathbf f\|_{L^{2}}+\|\mathbf g\|_{L^{2}}\big)^2
\\ & \\
& \displaystyle
\leq \ \alpha^{-1}\, \frac {C_P} {2}\, \|\nabla \yy - \mathbf R\|_{L^2}^2 \ +\ \alpha\,  {C_P} \,
\big(\|\mathbf f\|_{L^{2}}^2+\|\mathbf g\|_{L^{2}}^2\big) \,.\end{array}
\end{equation*}
By choosing $\,\alpha\,= \,h\, C_P\, C_g/C$
\begin{equation}\label{lowerbd2}
  \begin{array}{ll}
&\mathcal L(\yy-\x)\,=\, \mathcal L(\yy-\mathbf R\x-\mathbf c) \, +\, \mathcal L(\mathbf R\x-\mathbf x)  \,\le
\\
&\\
& \displaystyle
\leq \ \alpha^{-1}\, \frac {C_P} {2}\, \|\nabla \yy - \mathbf R\|_{L^2}^2 \ +\ \alpha\,  {C_P}  \,
\big(\|\mathbf f\|_{L^{2}}+\|\mathbf g\|_{L^{2}}\big)^2
\, +\, \mathcal L(\mathbf R\x-\mathbf x) \ =
\\ & \\
& \displaystyle
= \ h^{-1}\, \frac {C/C_G} {2}\, \|\nabla \yy - \mathbf R\|_{L^2}^2 \ +\ \frac {C_P^{\,2}} {C/C_G}\, h \,
\big(\|\mathbf f\|_{L^{2}}^2+\|\mathbf g\|_{L^{2}}^2\big) \, +\, \mathcal L(\mathbf R\x-\mathbf x)\,.\end{array}
\end{equation}
Exploiting the standard representation \eqref{eurod} of the rotation
$\mathbf R=\mathbf I+\mathbf W \sin\vartheta+(1-\cos\vartheta)\mathbf W^{2}$
for suitable $\vartheta\in\R$ and $\mathbf W\in\M^{N\times N}_{skew}$,
by
\eqref{globalequi},\eqref{comp},\eqref{lowerbd} and \eqref{lowerbd2} 
we get 
\begin{equation}\label{lowerbd3}\begin{array}{ll}
\displaystyle\F_{h}(\v) & \displaystyle
\ge\ \frac {C/C_G} {2} \, h^{-2}\int_{\Omega}|\nabla \yy-\mathbf R|^{2}\,d\x\,-\,\frac {C_P^{\,2}} {C/C_G}\,
\big(\|\mathbf f\|^{2}_{L^{2}}+\|\mathbf g\|^{2}_{L^{2}}\big)\,-\,h^{-1}\mathcal L\big((\mathbf R-\mathbf I)\x\big)\,\ge\\
&\\
&\ge\
\displaystyle
-\, \frac {C_P^{\,2}\,C_G} C \, \big(\|\mathbf f\|^{2}_{L^{2}}+\|\mathbf g\|^{2}_{L^{2}}\big)
\qquad \qquad \forall \vv \in H^1(\Omega;\R^N)\,, \ \forall h>0\,.
\end{array}
\end{equation}
\vskip-0.6cm
\hfill\end{proof}
\begin{lemma}
\label{corcurl}
Let
$\v_{n}\in H^1(\Omega;\R^N)$ be a sequence such that $\mathbb E(\v_n)\wconv \mathbf T$ in $L^2(\Omega;\M^{N\times N})$.
Then there exists $\w\in H^1(\Omega;\R^N)$ such that $\mathbf T=\mathbb E(\w)$.
 If in addition  $\nabla\v_{n}\wconv \mathbf G$ in $L^{2}(\om;\mathbf R^{N})$
then there exists a constant matrix $\mathbf W\in \M^{N\times N}_{skew}$ such that
$\nabla\w=\mathbf G-\mathbf W$.
\end{lemma}

\begin{proof} Since  $\mathbb E(\v_n)\wconv \mathbf T$ in $L^2(\Omega;M^{N\times N})$ then $\v_{n}-\mathbb P \v_n$ are equibounded in $H^{1}(\om; \mathbf R^{N})$, where $\mathbb P $  the projection on the set $\mathcal R$ of infinitesimal rigid displacements.
Therefore, up to subsequences, we can assume that  $\w_{n}:=\v_{n}-\mathbb P \v_n\wconv \w$ in $H^{1}(\Omega;\mathbf R^{N})$
and we get
\begin{equation}
\lab{dec}
\nabla\w_{n}=\mathbb E(\w_{n})+{\rm skew}\nabla\w_n=\mathbb E(\v_{n})+{\rm skew}\nabla\w_n
\end{equation}
hence there exists $\mathbf S\in L^2(\Omega;\M_{skew}^{N\times N})$ such that
${\rm skew}\nabla\w_n\wconv\mathbf S$ in $L^2(\Omega;\M^{N\times N})$ and by letting $n\to +\infty$ into \eqref{dec}
we have  $\mathbf T+\mathbf S=\nabla \w$.
 Since $\mathbf S\in L^2(\Omega;\M_{skew}^{N\times N})$ we get
$\mathbb E(\w)=\mathbf T$ and  if in addition
 $\nabla\v_{n}\wconv \mathbf G$ in $L^{2}(\om;\M^{N\times N})$  then there exists a  constant
$\mathbf W\in \M^{N\times N}_{skew}$
such that
$\nabla \mathbb P \v_{n}\wconv \mathbf W$ in $L^{2}(\om;\M^{N\times N})$,
actually converging in the finite dimensional space of constant skew symmetric matrices,
thus proving the Lemma.
\end{proof}

\begin{remark}
{\rm It is worth noting that if $\mathbb E(\v_j)\wconv \mathbf T$ in
$L^2(\Omega;\M^{N\times N})$ then by Lemma~\ref{corcurl} there exists $\v\in H^1(\Omega;\R^N)$ such that
$\mathbf T=\mathbb E(\v)$ and if $\mathbf T=\mathbb E(\w)$ for some $w\in H^1(\Omega;\R^N)$, then $\v-\w$
is an infinitesimal rigid displacement in $\Omega$, {\emph i.e.}
$\mathbb E(\v-\w)=\mathbf 0$.
}
\end{remark}
\vspace{0.5cm}
Next we show a preliminary convergence property: we compute a kind of Gamma limit
of the sequence of functional $\mathcal F_h$ with respect to weak $L^2$ convergence of linearized strains.
\begin{lemma}(\textbf{energy convergence})
\label{convfunc}
Assume  that  \eqref{globalequi}  holds true and let $h_{j}\to 0$ be a decreasing sequence. Then
\begin{itemize}
\item[i)] For every $\v_j,\v\in H^1(\Omega;\R^N)$ such that
$\mathbb E(\v_j)\wconv\mathbb E(\v)$ in $L^2(\Omega;\M^{N\times N})$ we have
$$\liminf_{j\rightarrow +\infty} \mathcal F_{h_{j}}(\v_j)\geq  \mathcal F(\v).$$\
\item[ii)] For every $\v\in H^1(\Omega;\R^N)$ there exists a sequence $\v_j\in H^1(\Omega;\R^N)$ such that $\mathbb E(\v_j)\wconv \mathbb E(\v)$ in $L^2(\Omega;\M^{N\times N})$ and
$$\limsup_{j\rightarrow +\infty}  \mathcal F_{h_{j}}(\v_j)\leq  \mathcal F(\v).$$
\end{itemize}
\end{lemma}
\begin{proof}
First we prove i).
We set $\yy_j=\x+h \v_j$ and denote various positive constants by $C, C',C'',...,L',L''$.
We may assume without restriction that $\mathcal F_{h_j}(\v_j)\leq C$;
 by taking into account \eqref{coerc} we get
$$Ch^{-2}\int_\Omega \vert\nabla\yy_j^T\nabla\yy_j- \mathbf I\vert^2\,d\x-
\mathcal L(\v_j) \ \le \
\mathcal F_{h_{j}}(\v_j)
$$
and by \eqref{globalequi}
$$h^{-2}\int_\Omega \vert\nabla\yy_j^T\nabla\yy_j- \mathbf I\vert^2\,d\x \ \leq \ C'\,+\,
\mathcal L(\v_j)\ =\ C'+\mathcal L(\v_j-\mathbb P\v_j),$$
where $\mathbb P\,\v_j$ is the projection of $\v_j$ onto the set of infinitesimal rigid displacements.
\\
Hence by Korn-Poincar\'{e} inequality we have
\begin{equation}
\label{bound}
h^{-2}\int_\Omega \vert\nabla\yy_j^T\nabla\yy_j- \mathbf I\vert^2\,d\x\ \leq\ C'\,+\,
C''\left(\int_{\Omega}\vert\mathbb E(\v_j)\vert^2\,d\xx\right)^{\frac{1}{2}}\leq\ C'''.
\end{equation}
Inequality \eqref{bound} together with the Rigidity Lemma of \cite{FJM}
imply, by the same computations at the beginning of proof of Lemma \ref{infFh},
that for every $h_{j}$ there exists a constant rotation $\mathbf R_j\in SO(N)$
and a constant $C'''$, dependent only on $\Omega$,
such that
$$\int_\Omega\vert\nabla\yy_j-\mathbf R_j\vert^2\,d\x\leq C''''h_{j}^2$$
that is
\begin{equation}
\label{bound1}
\int_\Omega\vert\mathbf I+h\nabla\v_j-\mathbf R_j\vert^2\,d\x\leq C''''h_{j}^2.
\end{equation}
Due to the representation \eqref{eurod} of rotations
for every $j\in \mathbb N$ there exist $\vartheta_j\in \mathbb R$ and
$\mathbf W_j\!\in \!\M^{N\times N}_{skew},\
\vert\mathbf W_{h_{j}}\vert^{2}=2$ such that
$\mathbf R_j=\exp(\vartheta_j\mathbf W_j)$ 
and 
\begin{equation}\lab{euform}
 \mathbf R_j\,=\,\exp(\vartheta_j \mathbf W_j)\,=\,
\mathbf I\,+\,\sin\vartheta_j\,\mathbf W_j\, +\,(1-\cos\vartheta_j)\,\mathbf W_j^{2}
\end{equation}
hence by \eqref{bound1}
\begin{equation}
\label{bound2}
\int_\Omega\vert h_j\nabla\v_j-\sin\vartheta_j \mathbf W_j-(1-\cos\vartheta_j)\mathbf W_j^2\vert^2\,d\x\leq C''''h_{j}^2\,.
\end{equation}
Since
$${\rm sym} \Big(h_j\nabla \v_j-\sin\vartheta_j \mathbf W_j
-(1-\cos\vartheta_j)\mathbf W_j^2\Big)=h_{j}\mathbb E(\v_j)-(1-\cos\vartheta_j)\mathbf W_j^2$$
 we get
$$\int_\Omega\vert \mathbb E(\v_j)-(1-\cos\vartheta_j)h_{j}^{-1}\mathbf W_j^2\vert^{2}\,d\x
\leq
{C''''}
.$$
By recalling that $\mathbb E(\v_j)\wconv\mathbb E(\v)$ in $L^2(\Omega;\M^{N\times N})$,
we deduce for suitable $L>0$
\begin{equation}
\label{bound3}
\left | 1-\cos\vartheta_j\right |=\frac{1}{2}\left |(1-\cos\vartheta_j)\mathbf W_j^2\right | \leq L {h_{j}}
\end{equation}
that is
\begin{equation}
\label{bound4}
\vert\sin \vartheta_j\vert\leq \sqrt{2Lh_{j}}.
\end{equation}
By \eqref{bound2}, \eqref{bound3} and \eqref{bound4} we obtain
\begin{eqnarray}
\label{convskewint}
\frac 1 2
\int_\Omega \!\vert && \!\!\!\!\!\!\!\!\!\!\!\!\!\!\!
\sqrt{h_{j}} \nabla\v_j|^2d\xx \leq \\
\nonumber
\,&\leq&
\int_\Omega \!\vert \sqrt{h_{j}}\nabla\v_j-h_{j}^{-1/2}\sin\vartheta_j \mathbf W_j
\vert^2d\x
+
\int_\Omega \!\vert h_{j}^{-1/2}\sin\vartheta_j \mathbf W_j
\vert^2d\x \leq \\
\nonumber
&\leq& (C'''+2L|\om|)\,h_{j} \ ,
\end{eqnarray}
hence, up to subsequences, by \eqref{bound4} there exists a constant matrix
$\mathbf W\in \M^{N\times N}_{skew}$ such that
\begin{equation}\label{proofconvsq}
  \sqrt{h_{j}}\nabla\v_j\ \rightarrow\ \mathbf W \qquad \hbox{strongly in } L^2(\Omega;\M^{N \times N})
\end{equation}
and therefore
\begin{equation}\label{proofconvhNablTNabl}
h_{j}\nabla\v_j^T\nabla\v_j\ \rightarrow\ \mathbf W^T\mathbf W\ =\ -\,\mathbf W^{2} \qquad
\hbox{strongly in } L^1(\Omega;\M^{N \times N})\,.
\end{equation}
By Lemmas 4.2 and 4.3 of \cite{DMPN}
for every $k \in \N$ there exist
an increasing sequence of Caratheodory functions
$\mathcal V^k_j : \om \times \M^{N \times N}_{sym} \to [0, +\infty)$
and a measurable function $\mu^k:\om \to (0,+\infty)$
such that $\mathcal V^k_j(\xx,\cdot)$ is convex for a.e. $\xx \in \om$ and satisfies
\beeq \lab{riscal}
	\mathcal V^k_j (\xx,\mathbf B) \ \leq\  \mathcal V(\xx,h_j \mathbf B) / h_j^2 \ =\
           \mathcal V_{h_j}(\xx,h_j \mathbf B)
	\qquad \forall\, \mathbf B \in
\M^{N \times N}_{sym} \, ,
\eneq
\beeq \lab{identica}
	\mathcal V_j^k (\xx,\mathbf B) \  = \ \Big( 1-\frac{1}{k} \Big) \, \mathcal V_0(\xx,\mathbf B)
	\qquad \mbox{for $\mathcal V_0(\xx,\mathbf B)\leq \mu^k(\xx) / h_j^{2}$ }.
\eneq
Property (\ref{identica}) entails
\beeq \lab{Vjconv}
	\lim_{j \to +\infty} \mathcal V^k_j (\xx, \mathbf B) =
	\Big( 1-\frac{1}{k} \Big) \,\mathcal V_0( \xx,\mathbf B) \,
            \qquad \hbox{a.e. }\xx\in\Omega,\  \forall\,\mathbf B \in \M^{N \times N}_{sym}
\eneq
then, by exploiting (\ref{Vjconv}), Lemma 4.3 of \cite{DMPN},
and taking into account that  \vskip0.2cm
\centerline{$\mathbf B_j:=\mathbb E(\v_j) + \textstyle{\frac12} h_{j} \nabla \v_{h_j}^T
	\nabla \v_{h_j}\wconv \mathbb E(\v)- \textstyle{\frac12}\mathbf W^{2}$ in $L^1(\Omega;\M^{N \times N})$\,,}
by \eqref{riscal} and \eqref{Vjconv} we get
\beea
	\liminf_{j \to 0} \int_{\om} \mathcal V_{h_j}
	(x,\nabla \v_{h_j}) \,d\xx
	\!\!& \geq &\!\!
	\liminf_{j \to +\infty} \int_{\om} \mathcal V_j^k
	(x, \mathbf B_j)\,d\xx
	\nonumber \\
	\!\!& \geq &\!\!
	\int_{\om} \!(1- 1/k) \mathcal V_0
\Big( \xx,\,\mathbb E(\v)- \textstyle{\frac12}\mathbf W^{2}\Big) \,d\xx \quad \forall k \in \N\,.
	\, \nonumber
\enea
Up to subsequences $\v_j-\mathbb P\v_j\wconv \w$ in $H^1(\Omega;\R^N)$, moreover $\mathbb E(\v)=\mathbb E(\w)$. Then by \eqref{globalequi} for every $k \in \N$ we obtain
 \beeq\begin{array}{ll}
 \displaystyle\liminf_{j\rightarrow +\infty} \mathcal F_{h_{j}}(\v_j)
 &\displaystyle\geq\ \int_{\om} \Big(1-\frac1k \Big) \mathcal V_0(\xx, \mathbb E(\v)- \textstyle{\frac12}\mathbf W^{2}) \,d\xx-\mathcal L(\w)=
 \vspace{0.1cm}
 \\
&=\ \displaystyle\int_{\om} \Big(1-\frac1k \Big) \mathcal V_0( \xx, \mathbb E(\v)- \textstyle{{\frac12}\mathbf W^{2}}) \,d\xx-\mathcal L(\v)\nonumber \\
\end{array}
\eneq
Taking the supremum as $k \to \infty$
we deduce
\beeq \label{liminf}
\displaystyle\liminf_{j\rightarrow +\infty} \mathcal F_{h_{j}}(\v_j)\geq \int_{\om}
\mathcal V_0( \xx, \mathbb E(\v)- \textstyle{{\frac12}\mathbf W^{2}}) \,d\xx-\mathcal L(\v)
\geq \mathcal F (\vv)
\eneq
which proves $i)$.
\vskip0.1cm
We are left to prove claim ii).
To this aim,
we set for every $\v\in H^{1}(\om;\mathbb R^{N})$:
\begin{equation}\label{Wv}
\displaystyle\mathbf W_{\v}\in \argmin \left\{\int_{\om}\mathcal V_0\Big( \xx,\, \mathbb E(\v)-\textstyle\frac{1}{2}\mathbf W^{2}\Big)\, d\xx: \ \mathbf W\in \M^{N\times N}_{skew}\right\}\,.
\end{equation}
Without relabeling, $\v$ denotes also a fixed compactly supported  extension in $H^{1}(\mathbb R^{N};\mathbb R^{N})$ of the given $\v$
(such extension exists since $\om $ is Lipschitz due to \eqref{OMEGA}).\\
We may define a recovery sequence $\w_{j}\in C^1(\overline\om;\mathbb R^{N})$ for every $j$, as follows. Set
\begin{equation}
\w_{j}=h_{j}^{-1/2}\,\mathbf W_{\v}\,\x+\v\star\varphi_{j}
\end{equation}
where the sequence $\eps_{j}$ is chosen in such a way that 
$h_{j}\eps_{j}^{-3}\to 0$ holds true and
$\varphi_{j}(\xx)=\varepsilon_j^{-N}\varphi(\xx/\varepsilon_j )$ is a 
mollifier supported in $B_{\eps_{j}}(\mathbf 0)$.
Sobolev embedding
entails $\v\in L^{6}(\R^N;\mathbb R^{N})$, since $\v\in H^{1}(\R^N;\mathbb R^{N})$ and $N=2,3$; then by Young Theorem and $0<\varepsilon_j\leq 1$ we have
\begin{equation}\lab{young}
\|\nabla(\v\star\varphi_{j})\|_{L^{\infty}}\le 
\|v\|_{L^{6}}\|\nabla\varphi_{j}\|_{L^{6/5}}\le \eps_{j}^{-N/6\,-1}\|\nabla\varphi\|_{L^{6/5}}\|v\|_{L^{6}}
\le \eps_{j}^{-3/2}\|\nabla\varphi\|_{L^{6/5}}\|v\|_{L^{6}}\,.
\end{equation}
By $\,\nabla \w_j= h^{-1/2}\mathbf W_{\v}+\nabla (\v\star\varphi_j)\,$ and $\,\mathbf W_{\v}^T=-\mathbf W_{\v}\,$ we get
 \begin{equation*}\begin{array} {ll}
 &\hskip-0.2cm \mathbb E(\w_{j})\,=\,\mathbb E(\v)\star\varphi_{j}\,,\vspace{0.2cm}\\
 &\hskip-0.2cm h_{j}\nabla\w_{j}^{T}\nabla \w_{j}\,=\,
-\mathbf W_{\v}^2\,+\, h_{j}\nabla(\v\star\varphi_{j})^{T}\,\nabla(\v\star\varphi_{j})\,+\,
h_{j}^{1/2}\big(\nabla(\v\star\varphi_{j})^{T}\,\mathbf W_{\v}-\mathbf W _{\v}\nabla(\v\star\varphi_{j})\big)\,,\\
\end{array}\end{equation*}
hence, by taking into account \eqref{young} and
$h_{j}\eps_{j}^{-3}\to 0$, we get
\beeq\lab{L2}\mathbb E(\w_{j})+\textstyle\frac{1}{2}h_{j}\nabla\w_{j}^{T}\nabla \w_{j}\to \mathbb E(\v)- \textstyle\frac{1}{2}\mathbf W_{\v}^2\quad \hbox{in} \ \ L^{2}(\om,\mathbb R^{N})\,,
\eneq
\beeq\lab{Linfty}
h_{j}\Big(\mathbb E(\w_{j})+\textstyle\frac{1}{2}h_{j}\nabla\w_{j}^{T}\nabla \w_{j}\Big) \ \to \ \mathbf 0\quad \hbox{in} \ \ L^{\infty}(\om,\mathbb R^{N})\,.
\eneq
Therefore Taylor's expansion of $\mathcal V$ entails
\begin{equation}
\lim_{j\to +\infty}\mathcal V_{h_{j}}(\xx, \mathbb E(\w_{j})+\textstyle\frac{1}{2}h_{j}\nabla\w_{j}^{T}\nabla \w_{j})=\mathcal V_0\left(\xx,\mathbb E(\v)-\textstyle{\frac 12}\mathbf W^2_{\v}\right)
\qquad \hbox{for a.e. }\xx\in \om,
\end{equation}and taking into account \eqref{upbound},\,\eqref{Wh},\,\eqref{Vh},\,\eqref{Linfty} we have
\beeq
\mathcal V_{h_{j}}(\xx, \mathbb E(\w_{j})+\textstyle\frac{1}{2}h_{j}\nabla\w_{j}^{T}\nabla \w_{j})\le C|\mathbb E(\w_{j})+\textstyle\frac{1}{2}h_{j}\nabla\w_{j}^{T}\nabla \w_{j}|^{2}\,,
\eneq
hence the Lebesgue dominated convergence theorem yields
$$ \mathcal F_{h_{j}}(\w_{j})\rightarrow\min_{\mathbf W\in \M^{N\times N}_{skew}}
\int_\Omega \mathcal V_0\left(\xx,\mathbb E(\v)-\textstyle{\frac 12}\mathbf W^2_{\v}\right)\,d\xx-
\mathcal L(\v) \,, $$
thus proving ii).
\end{proof}
\begin{remark}\rm{If $\mathcal W$ is a convex function of $\mathbf F^{T}\mathbf F-\mathbf I$ then \eqref{liminf} is a straightforward consequence of weak $L^1(\om;\M^{N\times N})$
convergence of $\mathbf{B}_j$,
hence introduction and use of integrands $\mathcal V^k_j$ in the proof is unnecessary in such case.
So restriction to decreasing sequences $h_j$
(used to apply Lemmas 4.2, 4.3 of \cite{DMPN})
can be deleted in the assumptions of Lemma \ref{convfunc} if $\mathcal W$ is convex.}
\end{remark}
%
\vskip0.2cm
Lemma \ref{infFh} entails existence of minimizing sequences for the sequence of functionals $\F_h$.
Next Proposition entails a (very weak) relative compactness property of these sequences;
in its proof we will consider the cone $ \mathbb K = \{ \tau(\Rot-\Id):\ \tau>0,\:\Rot\in SO(N)\} $
which fulfils,
thanks to \eqref{cone} and \eqref{clcone}:
\begin{equation}\lab{Kappa}
\overline{\mathbb K}\,=\,\mathbb K\cup \M^{N\times N}_{skew}.
\end{equation}
\begin{equation}\lab{Kappabis}
\overline{\mathbb K}+\M^{N\times N}_{skew}\,=\, {\mathbb K}+\M^{N\times N}_{skew} \,=\,
\left\{\mathbf W+\mathbf Z^2:\ \mathbf W,\,\mathbf Z\,\in\M^{N\times N}_{skew}\,\right\}\,.
\end{equation}
\begin{lemma}(\textbf{Compactness of minimizing sequences})
\label{compact}
 Assume that \eqref{globalequi}, \eqref{comp} hold true,
 $h_{j}\to 0$ is a sequence of strictly positive real numbers 
 and the sequence of displacements $\v_{{j}}\in H^1(\Omega;\R^N)$ fulfil $(\mathcal F_{h_{j}}(\v_{j})-\inf\mathcal F_{h_j})\to 0$. \\ Then there exists $C>0$ such that $\|\mathbb E(\v_{j})\|_{L^{2}}\le C$.
\end{lemma}
\begin{proof}
By Lemma \ref{infFh} there exists $c$ such that
\begin{equation}\label{estimates}
-\infty \,<\,c\,\leq\,\inf\mathcal F_{h_j}\,\leq\, \mathcal F_{h_j}(\mathbf 0)\,=\, 0.
\end{equation}
Assume by contradiction that $t_{j}:=\parallel \mathbb E(\v_{{j}})\parallel_{L^2}\rightarrow +\infty$ and set $\w_{j}=t_{j}^{-1}\v_{j}$. It is readily seen that by Lemma \ref{corcurl} there exists $\w\in H^1(\Omega;\R^N)$ and a subsequence such that without relabeling $\mathbb E(\w_j)\wconv\mathbb E(\w)$ in $L^2(\Omega;\M^{N\times N})$. \\
By \eqref{estimates} we can assume up to subsequences that $\F_{h_{j}}(\v_{{j}})\leq h_j^2$. \\ 
Moreover
by setting $\yy_{{j}}=\x+h_j\v_{{j}}=\x+h_{j}t_{j}\w_{j}$, arguing as at the beginning of Lemma~\ref{infFh} 
proof and exploiting Korn-Poincar\'e inequality \eqref{kornpoi}, we obtain that
for every $j\in\mathbb N$ there exists a constant rotation $\Rot_{j}\in SO(N)$ such that
$$\int_\Omega\vert\nabla\yy_{{j}}-\Rot_j\vert^2\,d\x\,\leq\, h_{j}^2\,+\, C_K(\|f\|_{L^2(\partial\om)}+\|g\|_{L^2(\om)}) \,t_j\,h_j^2,$$
that is, by setting $C\,'=C_K(\|f\|_{L^2(\partial\om)}+\|g\|_{L^2(\om)})$,
\begin{equation}
\label{conda}
\int_\Omega\vert \Id+h_jt_{j}\nabla\w_j-\Rot_j\vert^2\,d\x\,\leq\, h_{j}^2\,(1+C\,'\,t_j).
\end{equation}
By possible further extraction of subsequences one among the three alternatives take place:
\begin{equation*}
a)\ h_jt_{j}\rightarrow \lambda>0\,, \qquad b)\ h_jt_{j}\rightarrow 0\,, \qquad c)\ h_jt_{j}\rightarrow +\infty\,.
\end{equation*}
If condition $a)$ holds true we have
$$\int_\Omega\left |\nabla\w_j-\frac{\Rot_j-\Id}{h_jt_{j}}\right |^2\,d\x\,\leq\,
\frac{1}{t_j^{2}}\,+\,\frac {C\,'}{t_j}\,,$$
hence, up to subsequences,
\begin{equation}\lab{conva}
\nabla \w_j\rightarrow \frac{\Rot-\Id}{\lambda}
\end{equation}
strongly in $L^2(\Omega; \M^{N\times N})$ for a suitable constant matrix $\Rot\in SO(N)$ and by Lemma \ref{corcurl} we get $\nabla\w\in \mathbb K+\M^{N\times N}_{skew}$.
\vskip0.2cm
If condition $b)$ holds true, then by using formula \eqref{euform} and  \eqref{conda} there exist  $\vartheta_{h_j}\in[0,2\pi]$ and  a constant $\mathbf W_{j}\in \M^{N\times N}_{skew}$ with $\vert \mathbf W_{j}\vert^2=\vert \mathbf W_{j}^2\vert^2=2$
 such that
\begin{equation}
\label{condb}
\int_\Omega\vert h_jt_{j}\nabla\w_{{j}}-
\sin \vartheta_{h_j}\mathbf W_{j}-(1-\cos\vartheta_{j})\mathbf W_{j}^2\vert^2\,d\x\leq \,h_{j}^2\,(1+C\,'\,t_j)\,.
\end{equation}
Since $\mathbf W_{j}$ and $\mathbf W_{j}^2$ aree respectively skew-symmetric
and symmmetric, \eqref{condb} yields
\begin{equation}
\label{condb1}
\int_\Omega\left | \mathbb E(\w_j)-
\frac{(1-\cos\vartheta_{j})}{h_jt_{j}}\mathbf W_{j}^2\right |^2\,d\x\,\leq \,t_j^{-2} \,+\, C\,'\,t_j^{-1}
\end{equation}
and bearing in mind  that $\int_\Omega\vert \mathbb E(\w_j)\vert^2\,d\x\,=\,1\,$
we get
\beeq\lab{incos}
\left| \frac{(1-\cos\vartheta_{j})}{h_jt_{j}}\right|
\,=\, \frac 1 2 \,
\displaystyle\left| \frac {(1-\cos\vartheta_{j})} {h_jt_{j}} \mathbf W_{j}^2\right|
\,\leq\, C\,'',\eneq
hence
\beeq\lab{insen}
\left |\sin\vartheta_{h_j}\right |
\,\leq\,
\sqrt{2(1-\cos\vartheta_{h_j})}
\,\leq\,\sqrt{2\,C\,''\,h_j\, t_{j}}
\eneq
and by \eqref{condb}
\begin{equation}
\label{condb2}
\int_\Omega\left | \sqrt{h_jt_{j}}\nabla\w_j-\frac{\sin\vartheta_{j}}{\sqrt{h_jt_{j}}}\mathbf W_{j}-
\frac{(1-\cos\vartheta_{j})}{\sqrt{h_jt_{j}}}\mathbf W_{j}^{2}\right |^2\,d\x\,\leq\, C\,'''\,h_{j}\,.
\end{equation}
By \eqref{incos} we know that $1-\cos\vartheta_{j}=o(\sqrt{h_jt_{j}}) $, hence \eqref{insen},\eqref{condb2} entail, up to subsequences,
$\sqrt{h_jt_{j}}\nabla\w_j\rightarrow \mathbf W  \in\M^{N\times N}_{skew}$ strongly in
$L^2(\Omega;\M^{N\times N})$ and since by \eqref{coerc} and Poincar\'{e}-Korn inequality
%
\begin{eqnarray*}
t_j\int_\Omega\vert \mathbb E(\w_{j})+\textstyle\frac{1}{2}h_{j}t_{j}\nabla\w_j^T\nabla\w_j\vert^2\,d\xx
&\leq &
 C^{IV}+
\mathcal L(\w_j) \\
&=&
C^{IV}+\mathcal L(\w_j-\mathbb P\w_{j})\\
&\leq&
 C^{V}\left(\int_\Omega\vert \mathbb E(\w_j)\vert^2\,d\x\right)^{\frac{1}{2}},
\end{eqnarray*}
we deduce
$$\int_\Omega\vert \mathbb E(\w_j)+\textstyle\frac{1}{2}h_{j}t_j\nabla\w_j^T\nabla\w_{j}\vert^2\,d\xx\rightarrow 0.$$

On the other hand, since
$$\liminf_{j\rightarrow +\infty}\int_\Omega\vert 2\mathbb E(\w_j)+h_{j}t_j\nabla\w_j^T\nabla\w_j\vert^2\,d\x \geq \int_\Omega\vert 2\mathbb E(\v)-\mathbf W^2\vert\,d\x,$$
 we get $2\mathbb E(\w)=\mathbf W^2$ which implies
$\nabla\w={\rm skew}\nabla\w+\frac{1}{2}\mathbf W^2$,
hence ${\rm skew}\nabla\w$ is a gradient field, that is a constant skew-symmetric matrix. By taking
$$\Rot:=\Id+\frac{1}{2}\mathbf W^2+\frac{\sqrt 3}{2}\mathbf W$$
and by applying formula \eqref{euform} we get $\Rot \in SO(N)$
that is $\nabla\w-(\Rot-\Id)\in \M^{N\times N}_{skew}$
which implies $\nabla\w\in\mathbb K+ \M^{N\times N}_{skew}$ whenever condition $b)$ holds.
\vskip0.2cm
Eventually if condition $c)$ holds true, by  \eqref{conda}  we get
\beeq\lab{convc}
\int_{\om}\left |\nabla\w_j-\frac{\Rot_j-\Id}{h_{j}t_j}\right |^2\,d\x\,\leq\,
\frac{1}{t_j^{2}}\,+\,\frac {C\,'}{t_j}\,,
\eneq
and by taking into account that
$|\Rot_j-\Id|= o(h_{j}t_j)$ due to \eqref{conda}, we have $\nabla\w_{j}\rightarrow \mathbf 0$ strongly in $L^2(\Omega; M^{N\times N})$, hence  $\nabla\w\in \overline{\mathbb K}+\M^{N\times N}_{skew}$ still by Lemma \ref{corcurl}.
\vskip0.3cm
By summarizing, in all three cases if $t_{j}:=\parallel \mathbb E(\v_{{j}})\parallel_{L^2}\rightarrow +\infty$  and $\mathcal F_{h_{j}}(t_j\v_{{j}})\leq C$ then $\nabla(t_{j}^{-1}\v_{{j}})=\nabla\w_{j}\to \nabla\w$
in $L^2(\Omega;\M^{N\times N})$
and $\nabla\w\in \overline{\mathbb K}+\M^{N\times N}_{skew}$. \\
Therefore
$\,\mathbb E(\w_{j})\to \mathbb E(\w)$ in $L^2(\Omega;\M^{N\times N})\,$.\\
Since $\widetilde\w_j:=\w_{{j}}-\mathbb P\w_{{j}}$ are equibounded in $H^{1}(\om;\mathbf R^{N})$, every subsequence of  $\widetilde\w_j$ has a weakly convergent subsequence and if $\widetilde\w$ is one of the limits we get $\mathbb E(\widetilde\w)=\mathbb E(\w)$ hence  by \eqref{globalequi}  $\mathcal L(\widetilde\w)=\mathcal L(\w)$.
Therefore  every subsequence of $\mathcal L(\widetilde\w_j)$ has a subsequence which converges to $\mathcal L(\w)$ that is the whole sequence $\mathcal L(\widetilde\w_j)$ converges
to $\mathcal L(\w)$, hence
\begin{equation}\label{Lw}
\displaystyle
-\mathcal L(\w)
\,=\,
-\limsup_{j\rightarrow +\infty} \mathcal L(\widetilde\w_j)
\,=\,
-\limsup_{j\rightarrow +\infty} \mathcal L(\w_{{j}})
\,\leq\,
\liminf_{j\rightarrow +\infty} t_j^{-1} \mathcal F_{h_{j}}(\v_{{j}})
\,
.
\end{equation}
Since \eqref{estimates} 
entails
$\limsup t_{j}^{-1}F_{h_{j}}(\v_{{j}})\le 0$, by \eqref{Lw} we get $\,\mathcal L(\w)\,\ge\, 0\,.$ \\
By taking into account that $\nabla\w\in \mathbb K+\M^{N\times N}_{skew}$ then, either $\,\nabla\w\in \M^{N\times N}_{skew}\,$ or
$$\w(\x)=\tau(\Rot-\Id)\x+\mathbf A\x+\mathbf c\,,\hbox{ for some }\tau>0\,, \ \Rot\in SO(N)\,,
\ \Rot\neq\Id,\  \mathbf A\in\M^{N\times N}_{skew}\,, \ \mathbf c\in \R^N.$$ The second case cannot occur
since in such case by \eqref{eurod} there would exist $\vartheta\in \mathbb R$ with  $\cos\vartheta <1$ and $\mathbf W\in \M^{N\times N}_{sym},\ \mathbf W\not\equiv \mathbf 0$ such that $\Rot=\Id+(1-\cos\vartheta)\mathbf W^{2}+(\sin\vartheta)\mathbf W\in SO(N)$ hence
\eqref{globalequi}, \eqref{comp}
would entail the contradiction below
\begin{equation}\begin{array}{ll}
 \mathcal L(\w)&=\ \displaystyle \,\tau\int_{\partial\Omega}\mathbf f\cdot(\Rot-\Id)\x\,d\H^{N-1}+\tau\int_{\om}\mathbf g\cdot(\Rot-\Id)\x\,=\,\\ \vspace{0.2cm}
 &=\ \displaystyle\,\tau (1-\cos\vartheta)\int_{\partial\Omega}\mathbf f\cdot\mathbf W^{2}\x\,d\H^{N-1}+\tau (1-\cos\vartheta)\int_{\om}\mathbf g\cdot\mathbf W^{2}\x\, < \,0 \,.
\end{array}
\end{equation}
Hence $\nabla\w\in \M^{N\times N}_{skew}$ that is  $\E(\w)=\mathbf 0$ which is again a contradiction since $\|\mathbb E(\w_{j})\|_{L^{2}}=1$ and $\mathbb E(\w_{j})\to \mathbb E(\w)$ in $L^2(\Omega; M^{N\times N})$.
\end{proof}
\textsl{Proof of Theorem}~\ref{mainth1} -
First we notice that minimizing sequences for $\mathcal F_{h_j}$ do exist
$h_j$ for every sequence of positive real numbers converging to $0$,
thank to Lemma \ref{infFh}.\\
Fix a sequence of real numbers $h_j>0$ converging to $0$ and a minimizing sequence $\v_{j}$ for $\F_{h_{j}}$.\\
Up to a preliminary extraction of a subsequence we can assume that $h_j$ is decreasing.\\
Since $-\infty < \inf \F_{h_{j}}\le 0$ there is $C>0$
such that $\F_{h_{j}}(\v_{j})\le C$, hence
by Lemmas \ref{compact} and
\ref{corcurl}
$$\|\mathbb E(\v_{j})\|_{L^{2}}\le C$$
and there exists $\v_{0}\in H^1(\Omega;\R^N)$ such that,
up to subsequences, $\mathbb E(\v_{{j}})\wconv \mathbb E(\v_{0})$ in $L^2(\Omega;\M^{N\times N})$,
thus proving \eqref{wstr}.
By Lemma \ref{convfunc} we get
$$\liminf_{j\rightarrow +\infty} \mathcal F_{h_{j}}(\v_{{j}})\ \geq \ \mathcal F(\v_{0})$$
and again by Lemma \ref{convfunc}, for every $\v\in H^1(\Omega;\R^N)$ there exists $\widetilde\v_{{j}}\in H^1(\Omega;\R^N)$ such that $\mathbb E(\widetilde\v_{{j}})\wconv \mathbb E(\v)$ in $L^2(\Omega;\M^{N\times N})$ and
$$\limsup_{j\rightarrow +\infty}  \mathcal F_{{j}}(\widetilde\v_{{j}})\ \leq \ \mathcal F(\v).$$
Hence
\begin{equation}
\mathcal F(\v_{0})\,\le \,\liminf_{n\rightarrow +\infty} \mathcal F_{{j}}(\v_{j})\,\le\, \liminf_{n\rightarrow +\infty}(\inf\F_{h_{j}}+o(1))\,\le\, \limsup_{j\rightarrow +\infty}  \mathcal F_{h_{j}}(\widetilde\v_{j})\,\leq\,  \mathcal F(\v)
\end{equation}
 and  \eqref{min} is proven. 
\\
Eventually we notice that  \eqref{convsqrth} 
is a straightforward consequence of
\eqref{bound2}, \eqref{bound3}, \eqref{bound4} and \eqref{proofconvsq}
in the proof of Lemma \ref{convfunc}, so we are left only to prove \eqref{effe0}.\\ To this aim it will be enough to notice that by
\eqref{wstr} and  \eqref{convsqrth} 
we get
 \begin{equation*}
\mathbb E(\v_{j})+\textstyle\frac{1}{2}h_{j}\nabla\v_{j}^{T}\nabla\v_{j}\ \wconv \ \mathbb E(\v_{0})-\textstyle\frac{1}{2}\mathbf W_{0}^{2}\qquad\hbox{in} \ L^{1}(\om; \M^{N\times N})
\end{equation*}
and by recalling \eqref{min} and \eqref{liminf} we get
\begin{equation*}
\mathcal F(\v_{0})\ = \ \liminf_{j\rightarrow +\infty} \mathcal F_{h_{j}}(\v_{j})\ \ge \ \int_{\om}\mathcal V_{0}(\xx, \mathbb E(\v_{0})-\textstyle\frac{1}{2}\mathbf W_{0}^{2})\,d\xx\ \ge \ \mathcal F(\v_{0})\ ,
\end{equation*}
thus proving \eqref{effe0}.
\qed
\section{Limit problem and Linear Elasticity}\label{sectionlimpb}
We denote by $\mathcal E: H^1(\Omega;\R^N)\to \mathbf R$ the energy functional of classical linear elasticity
\begin{equation}
\label{linearel}
{\mathcal  E}(\v):=\int_\Omega \mathcal V_0(\xx,\mathbb E(\v))\,d\x-\mathcal L(\v)\,.
\end{equation}
Notice that \eqref{Eclassic} is just a particular model case of \eqref{linearel} corresponding to \eqref{WquadIntrod}.\\
As it was already emphasized the inequality $\mathcal  F \leq \mathcal  E$ always holds true.
Moreover the two functionals cannot coincide: indeed $\mathcal F (\v) < \mathcal E(\v)$ whenever $\v(\xx)=\frac 1 2 \mathbf W^2 \xx$
with $\mathbf W\in \mathcal M^{N\times N}_{skew}$.
However we can show that the two functionals $\mathcal  F $ and $\mathcal E$, notwithstanding their differences, have the same minimum and same set of minimizers when the loads are equilibrated and compatible, say when the load fulfils both \eqref{globalequi} and \eqref{comp}.\\
Next results clarify the relationship between the minimizers of classical linear elasticity functional $\EE$
and the minimizers of functional $\FF$ defined in
\eqref{DTfuncintro}
, which is the variational limit
of nonlinear energies $\FF_h$ in the sense shown by Theorem~\ref{mainth1}.
\begin{theorem}
\label{linel}
Assume that \eqref{globalequi} and \eqref{comp} hold true. Then
\beeq\lab{equalmin}
 \min_{\v\in H^1(\Omega;\R^N)} \mathcal F(\v)\ =\ \min_{\vv\in H^1(\Omega;\R^N)}\mathcal E(\vv).
\eneq
and
\begin{equation}
\label{equivmin}
\argmin_{\v\in H^1(\Omega;\R^N)}  \mathcal F\ =\ \argmin_{\v\in H^1(\Omega;\R^N)}{\mathcal E}.
\end{equation}
\end{theorem}
\begin{proof}
Both functionals $\mathcal F,\,\mathcal E$ do have minimizers under conditions \eqref{globalequi}, \eqref{comp}:
$\mathcal E$ by classical results and $\mathcal F$ by Theorem \ref{mainth1}.
Taking into account that
$\mathcal F(\v)\leq\mathcal E(\v)$ for every $\v\in H^1(\Omega;\R^N)$, and setting
$\,\mathbf \z_{\mathbf W}(x):=\textstyle\frac{1}{2}\mathbf W^{2}\x$
for every $\mathbf W\in \M_{skew}^{N\times N}$ \,, we get
$\, \E(\z_{\mathbf W})= \frac{1}{2}\mathbf W^{2}\,$ and
\beeq
\label{minFvsminE}
\begin{array}{ll}
&\displaystyle
\min_{\v\in H^1(\Omega;\R^N)} \mathcal E(\v) \ \geq \
\min_{\v\in H^1(\Omega;\R^N)} \mathcal F(\v)=
\\
&\displaystyle
\min_{\v\in H^1(\Omega;\R^N)}\left\{\min_{\mathbf W\in \M^{N\times N}_{skew}}\left\{\displaystyle\int_\Omega \mathcal V_0( \xx, \,\mathbb E(\v)-\textstyle\frac{1}{2}\mathbf W^{2})\, d\xx- \mathcal L(\v) \right\}\right\}=
\\
& \\
&\displaystyle
\min_{\mathbf W\in \M^{N\times N}_{skew}}\left\{\min_{\v\in H^1(\Omega;\R^N)}\left\{\displaystyle\int_\Omega \mathcal V_0(\xx, \,\mathbb E(\v)-\textstyle\frac{1}{2}\mathbf W^{2})\, d\xx- \mathcal L(\v) \right\}\right\}=\\
&\\
&\displaystyle\min_{\mathbf W\in \M^{N\times N}_{skew}}\left\{\min_{\v\in H^1(\Omega;\R^N)}\left\{\displaystyle\int_\Omega \mathcal V_0(\xx, \,\mathbb E(\v-\z_{\mathbf W}))\, d\xx- \mathcal L(\v-\z_{\mathbf W})-\mathcal L(\z_{\mathbf W}) \right\}\right\}=\\
&\\
&\displaystyle\min_{\z\in H^1(\Omega;\R^N)}\mathcal E(\z)\ \,-\max_{\mathbf W\in \M^{N\times N}_{skew}}
\mathcal L(\z_{\mathbf W})
\ \geq \ \min_{ H^1(\Omega;\R^N)}\mathcal E\,.
\end{array}
\end{equation}
where last inequality follows by $\mathcal L(z_{\mathbf W})\leq 0$, due to \eqref{comp}.
Therefore \eqref{equalmin} is proved and we are left to show \eqref{equivmin}.\\
First assume $\v\in \argmin_{\v\in H^1(\Omega;\R^N)} \mathcal F$ and let
\begin{equation}\label{Wv}
\displaystyle\mathbf W_{\v}\in \argmin \left\{\int_{\om}\mathcal V_0\Big( \xx,\, \mathbb E(\v)-\textstyle\frac{1}{2}\mathbf W^{2}\Big)\, d\xx: \ \mathbf W\in \M^{N\times N}_{skew}\right\}.
\end{equation}
If $\mathbf W_{\v}\neq \mathbf 0$ then, by setting $\z_{\mathbf W_{\v}}=\frac{1}{2}\mathbf W_{\v}^{2}\,\xx$
we get $\mathbb E(\z_{\mathbf W_\v})=\nabla\z_{\mathbf W_\v}= \frac{1}{2}\mathbf W_\v^{2}$ and, by compatibility \eqref{comp} we obtain
\begin{equation}\label{6.6}
\begin{array}{ll}
&\displaystyle\ \min \mathcal F= \mathcal F(\v)= \int_{\om}\mathcal V_0\Big(\xx,\,\mathbb E(\v-\textstyle\z_{\mathbf W_{\v}})\Big)\, d\xx-\mathcal L(\v- \z_{\mathbf W_{\v}})-\mathcal L(\z_{\mathbf W_{\v}})\,= \vspace{0.2cm}\\
&\displaystyle
\mathcal E (\v- \z_{\mathbf W_{\v}})\,-\, \mathcal L(\z_{\mathbf W_{\v}})
\geq\ \min\, \mathcal E\,-\, \mathcal L(\z_{\mathbf W_{\v}}) \,>\,  \min\, \mathcal E\,,
\end{array}
\end{equation}
say a contradiction. Therefore $\mathbf W_{\v}=\mathbf 0$, $\z_{\mathbf W_{\v}}=\mathbf 0$,
and all the inequalities in \eqref{6.6} turn out to be equalities, hence we get $\mathcal F(\v)=\mathcal E(\v)=\min \mathcal E=\min \mathcal  F$,
say $\v\in \argmin_{ H^1(\Omega;\R^N)}\mathcal E$
and $\argmin_{ H^1(\Omega;\R^N)}\mathcal F \subset \argmin_{ H^1(\Omega;\R^N)}\mathcal E.$
\\ In order to show the opposite inclusion, we assume $\v\in \argmin_{\v\in H^1(\Omega;\R^N)} \mathcal E$ and still referring to the choice \eqref{Wv} we set $\z_{\mathbf W_\v}=\frac 1 2 \mathbf W_\v^2\,\xx$. Then
\begin{eqnarray}\label{contrad}
   \mathcal F(\v)
   \!\!&=&\!\!
   \int_{\om}\mathcal V_0\Big( \xx,\,\mathbb E(\v-\textstyle\z_{\mathbf W_{\v}})\Big)\, d\xx-\mathcal L(\v- \z_{\mathbf W_{\v}})-\mathcal L(\z_{\mathbf W_{\v}})\,=\\
   \nonumber \vspace{0.1cm}
   \!\!&=&\!\!
   \mathcal E(\v- \z_{\mathbf W_{\v}})-\mathcal L(\z_{\mathbf W_{\v}})
   \,\geq\,
   \mathcal F(\v- \z_{\mathbf W_{\v}})-\mathcal L(\z_{\mathbf W_{\v}})\,.
\end{eqnarray}
This leads to the contradiction $ \mathcal F(\v) > \mathcal F(\v-\z_{\mathbf W_{\v}})$
if $\z_{\mathbf W_\v}\neq \mathbf 0$, due to \eqref{comp}; therefore $\z_{\mathbf W_\v} = \mathbf 0$ and we have equalities in place of inequalities in \eqref{contrad}: therefore $\mathcal E(\v)=\mathcal F(\v)$ and $\v\in \argmin_{ H^1(\Omega;\R^N)}\mathcal F$.
\end{proof}
\begin{corollary} Assume the standard structural assumptions, \eqref{globalequi},\,\eqref{comp}
and $\mathbf W_{0}\in \M^{N\times N}_{skew}$ is the matrix whose existence
is warranted by Theorem \ref{mainth1}. \ Then $\mathbf W_{0}=\mathbf 0$.
\end{corollary}
\begin{proof} Let $ \v_{0}$ be in $\argmin \F$, $\mathbf W_{0}
$ be the skew symmetric matrix in the claim of Theorem \ref{mainth1}
and assume by contradiction that $\mathbf W_{0}\neq \mathbf 0$. By \eqref{equalmin} and \eqref{equivmin} we get
\begin{equation*}
 \int_{\om}\mathcal V_{0}(\xx, \mathbb E(\v_{0})-\textstyle\frac{1}{2}\mathbf W_{0}^{2})\,d\xx\,
 =\,
 \displaystyle\int_{\om}\mathcal V_{0}(\xx, \mathbb E(\v_{0}))\,d\xx\,,
\end{equation*}
so by taking into account that $\mathcal V_{0}$ is a positive definite quadratic form
\begin{equation*}
\int_{\om}\mathcal V_{0}(\xx,\textstyle\frac{1}{2}\mathbf W_{0}^{2})\,d\xx-\displaystyle\frac{1}{2}\int_{\om}D\mathcal V_{0}(\xx, \mathbb E(\v_{0}))\cdot \mathbf W_{0}^{2}
\,=\,0.
\end{equation*}
Since by \eqref{equivmin} $\v_{0}\in \argmin \mathcal E$, the Euler-Lagrange equation yields
\begin{equation*}
\int_{\om}D\mathcal V_{0}(\xx, \mathbb E(\v_{0}))\cdot \mathbf W_{0}^{2}\,=\,\mathcal L(\mathbf W_{0}^{2}\xx)\,,
\end{equation*}
hence
\begin{equation*}
0 \leq \int_{\om}\mathcal V_{0}(\xx,\textstyle\frac{1}{2}\mathbf W_{0}^{2})\,d\xx\,=\displaystyle\,\frac 12 \,\mathcal L(\mathbf W_{0}^{2}\xx) < 0
\end{equation*}
which is a contradiction by \eqref{comp}.
\\ Hence $\mathbf W_{0}=\mathbf 0$.
\end{proof}
If strong inequality in \eqref{comp} is replaced by a weak inequality, then Theorem \ref{linel} cannot hold as it is, nevertheless
a weaker claim still holds true as it is shown by the following general result.
\begin{proposition} \lab{exinfmin}If the structural assumptions together with \eqref{globalequi} are fulfilled, but \eqref{comp} is replaced by
\begin{equation}
\lab{wcompbis} \mathcal L( \mathbf W ^2 \xx
)\le 0\qquad \forall\, \mathbf W\in \M^{N\times N}_{skew}
\end{equation}
then $\argmin \mathcal F$ is still nonempty and
\beeq\lab{eqmin}
\min\F=\min\mathcal E\,,
\eneq
but the coincidence of minimizers sets is replaced by the inclusion
\begin{equation}\lab{incl}
\argmin\mathcal E\subset\argmin\F\,.
\end{equation}
If \eqref{wcompbis} holds true and there exists $\mathbf U\in \M^{N\times N}_{skew},\ \mathbf U\neq \mathbf 0$ such that $\mathcal L(\mathbf U^2 \xx)=0$,
 then $\F$ admits infinitely many minimizers which are not minimizers of $\mathcal E$, precisely
\begin{equation}\label{argmin=}
\argmin \mathcal E\ \,\mathop{\subset}_{\neq} \ \, \argmin \mathcal E \, +\,  \left\{ \,\mathbf U^2\xx\, : \  \mathbf U \in \mathcal M^{N\times N}_{skew},\ \mathcal L(\mathbf U^2\xx)=0 \right\}
\ \,\mathop{\subset} \ \,\argmin \mathcal F \,,
\end{equation}
where the last inclusion is an equality in 2D:
\begin{equation}\label{argmin==}
\argmin \mathcal E\ \mathop{\subset}_{\neq} \ \argmin \mathcal E \, +\,  \left\{ \,-\, t\,\xx\, : \  t\geq 0 \right\}
\ = \ \argmin \mathcal F \,,\qquad \hbox{ if } N=2\,.
\end{equation}
\end{proposition}
\begin{proof}
The set $\argmin \mathcal E$ is nonempty by classical arguments.
Fix $\v_{*}\in \argmin \mathcal E$. Then for every $\v\in H^{1}(\om;\mathbf R^{N})$ and for every  $\mathbf W\in \M^{N\times N}_{skew}$, by setting $\z_{\mathbf W}=\frac 1 2 \mathbf W^2\xx$, we get
\begin{equation}\begin{array}{ll}
&\displaystyle\F(\v_{*})\le \mathcal E(\v_{*})\le  \mathcal E(\v-\z_{\mathbf W})=\int_{\om}\mathcal V_{0}(x, \mathbb E(\v)-\textstyle\frac{1}{2}\mathbf W^{2})\,dx- \mathcal L(\mathbf v-\z_{\mathbf W})\le\\
&\\
&\displaystyle\le\int_{\om}\mathcal V_{0}(x, \mathbb E(\v)-\textstyle\frac{1}{2}\mathbf W^{2})\,dx- \mathcal L(\mathbf v)
\end{array}
\end{equation}
hence for every $\v\in H^{1}(\om;\mathbf R^{N})$
\begin{equation}
\F(\v_{*})\le\displaystyle\min_{\mathbf W\in \M^{N\times N}_{skew}}\int_\Omega \mathcal V_{0}(x, \mathbb E(\v)-\textstyle\frac{1}{2}\mathbf W^{2})\, dx- \mathcal L(\v)=\F(\v)
\end{equation}
thus proving that $\argmin \mathcal F$ is nonempty and \eqref{incl}.\\
Moreover by setting
 \begin{equation}
\displaystyle\mathbf W_{\v}\in \argmin \left\{\int_{\om}
\mathcal V_{0}\big(x, \mathbb E(\v)-\textstyle\frac{1}{2}\mathbf W^{2}\big)\, dx:
\ \mathbf W\in \M^{N\times N}_{skew}\right\},
\qquad
\forall\ \v\in H^{1}(\om;\mathbf R^{N})
\end{equation}
condition \eqref{wcompbis} entails
\begin{equation}
\F(\v_{*})=\displaystyle\int_\Omega \mathcal V_{0}\big(x, \mathbb E(\v_{*}-\z_{\mathbf W_{\v_{*}}})\big)\, dx- \mathcal L(\v_{*})=\mathcal E(\v_{*}-\z_{\mathbf W_{\v_{*}}})-\mathcal L(\z_{\mathbf W_{\v_{*}}})\ge \mathcal E(\v_{*})
\end{equation}
hence \eqref{eqmin} follows by $\mathcal F\leq \mathcal E$.\\
If \eqref{wcompbis} holds true, $\mathcal L(\z_{\mathbf U})=0$ for some $\mathbf 0 \neq \mathbf U\in \M^{N\times N}_{skew}$ and $\vv^*\in\! \argmin \mathcal E$ then, by comparing the finite dimensional minimization over $\mathbf W$ with evaluation at $\mathbf W = \mathbf U$ and exploiting \eqref{eqmin}, we get 
\begin{equation}\begin{array}{ll}
&\F(\v_{*}+\z_{\mathbf U})=\displaystyle \min_{\mathbf W}\int_\Omega \mathcal V_{0}\big(x, \mathbb E(\v_{*})+\textstyle\frac{1}{2}\mathbf U^{2}-\textstyle\frac{1}{2}\mathbf W^{2}\big)\, dx- \mathcal L(\v_{*}+\z_{\mathbf U})\le
\vspace{0.15cm}\\
&\displaystyle\le \int_{\om}\mathcal V_{0}(x, \mathbb E(\v_{*}))\,dx-\mathcal L(\v_{*})
= \mathcal E(\v_{*})  = \min \mathcal E = \min \mathcal F \,,
\end{array}
\end{equation}
that is $\v_{*}+\z_{\mathbf U}\in \argmin\F$ . \,Since $\mathcal V_{0}$ is strictly convex we get
$\argmin\mathcal E=\{\v_{*}+\z: \z\in\mathcal R\}$ hence $\mathcal E(\v_{*}+\z_{\mathbf U})> \mathcal E(\v_{*})$
thus proving the strict inclusion in \eqref{argmin=}.\\
Concerning last claim, if \eqref{wcompbis} holds true, $\mathcal L(\z_{\mathbf U})=0$ for some $\mathbf 0 \neq \mathbf U\in \M^{N\times N}_{skew}$, $\vv^*\in\! \argmin \mathcal F$ and $N=2$,  then $\M^{2\times 2}_{skew}$ is a 1D space, therefore we can assume $\mathbf U=(\e_1\otimes \e_2 - \e_2\otimes \e_1 )$, $\mathbf U^2=-\mathbf I$, $\M^{2\times 2}_{skew}=
\span \mathbf U$
and $\mathbf W_{\vv_*}=\lambda \mathbf U$
for some $\lambda\in \R$, and by \eqref{eqmin}
\begin{eqnarray*}
  \min \mathcal E \,=\, \min \mathcal F \,=\, \mathcal F(\vv*) \,=\, \int_{\Omega}\mathcal V_0\big(\mathbb E(\vv*)-\frac 1 2 \mathbf W_{\vv_*}^2\big)\,d\xx -\mathcal L (\vv*)\,=\\
  =\,\int_{\Omega}\mathcal V_0\big(\mathbb E(\vv_*)-\frac {\lambda^2} 2 \mathbf U^2\big)\,d\xx -\mathcal L (\vv_*-\zz_{\mathbf \lambda U})\,=\,\mathcal E(\vv_*-\zz_{\mathbf \lambda U}) \,,\
\end{eqnarray*}
that is $(\vv_*\!-\!\zz_{\mathbf \lambda U})\in \argmin \mathcal E $ for every $\vv_*\!\in \!\argmin \mathcal F$,
therefore we get \vskip0.1cm
\centerline{$\argmin \mathcal F -\left\{ \,\zz_{\lambda \mathbf U}\, : \  \lambda\in \R \right\} \subset \argmin \mathcal E$\,, \quad
$\argmin \mathcal F \subset \argmin \mathcal E + \left\{ \,\zz_{\lambda \mathbf U}\, : \  \lambda\in \R \right\} $\,,}
\vskip0.1cm
hence by $z_{\lambda \mathbf U}=\frac {\lambda^2}2 \mathbf U^2\xx=-\frac {\lambda^2}2 \xx$ we obtain the equality in place of the last inclusion in \eqref{argmin=}, hence \eqref{argmin==}.
\end{proof}
Next example depicts the above Proposition in a simple explicit case.
\begin{example}\lab{infmany}
{\rm Let $\om=(-1/2,1/2)^{2},\  \mathbf g\equiv \mathbf 0,\ \mathbf f= (\mathbf 1_{S_{+}}-\mathbf 1_{S_{-}})\mathbf e_{2}+(\mathbf 1_{T_{+}}-\mathbf 1_{T_{-}})\mathbf e_{1}$ where $S_{\pm}$
denote respectively the right and the left side, and $T_{\pm}$ the upper and the lower side of the square
(see Fig.\ref{FigEx6.5}).
\vskip0.5cm
\begin{center}
\begin{tikzpicture}[scale=1.1]
 %
 \begin{scope}
 \draw [line width=0.08cm,red] [-] (-4.04,-1) -- (-0.96,-1);
 \draw [line width=0.08cm,red] [-] (-4.04,+2) -- (-0.96,+2);
 \draw [line width=0.08cm,red] [-] (-4.,-1) -- (-4.,+2);
 \draw [line width=0.08cm,red] [-] (-1.,-1) -- (-1.,+2);
 \end{scope}
\begin{scope}[very thick]
 \draw[->] (-3.7,2.3) -- (-3.3,2.3);
 \draw[->] (-3.2,2.3) -- (-2.8,2.3);
 \draw[->] (-2.7,2.3) -- (-2.3,2.3);
 \draw[->] (-2.2,2.3) -- (-1.8,2.3);
 \draw[->] (-1.7,2.3) -- (-1.3,2.3);
 \draw[->] (-3.3,-1.3) -- (-3.7,-1.3) ;
 \draw[->] (-2.8,-1.3) -- (-3.2,-1.3) ;
 \draw[->] (-2.3,-1.3) -- (-2.7,-1.3) ;
 \draw[->] (-1.8,-1.3) -- (-2.2,-1.3) ;
 \draw[->] (-1.3,-1.3) -- (-1.7,-1.3) ;
 \draw[->] (-4.3,-0.3) -- (-4.3,-0.7);
 \draw[->] (-4.3, 0.2) -- (-4.3,-0.2);
 \draw[->] (-4.3, 0.7) -- (-4.3, 0.3);
 \draw[->] (-4.3, 1.2) -- (-4.3, 0.8);
 \draw[->] (-4.3, 1.7) -- (-4.3, 1.3);
 \draw[->] (-0.7,-0.7) -- (-0.7,-0.3);
 \draw[->] (-0.7,-0.2) -- (-0.7, 0.2);
 \draw[->] (-0.7, 0.3) -- (-0.7, 0.7);
 \draw[->] (-0.7, 0.8) -- (-0.7, 1.2);
 \draw[->] (-0.7, 1.3) -- (-0.7, 1.7);
\end{scope}
\end{tikzpicture}
\end{center}
\vskip-0.4cm
\begin{figure}[h]
   \caption{Example \ref{infmany}.}
   \label{FigEx6.5}
\end{figure}
A straightforward computation gives, for suitable $\lambda\in\R$,
\begin{equation}\lab{nullbd}
\int_{\partial\Omega}\mathbf f\cdot\mathbf W^{2}\,\x\,d\H^{N-1}=
-\lambda^2\int_{\partial\Omega}\mathbf f\cdot\x\,d\H^{N-1}=0
\qquad
\forall\ \mathbf W\in \M^{2\times 2}_{skew} \,,
\end{equation}
Then,  since \eqref{globalequi} and \eqref{wcompbis} are fulfilled, by \eqref{argmin==} in  Proposition \ref{exinfmin}, we know that, for every choice of $\mathcal V_{0}$ satisfying the standard structural hypotheses,
  $\F$ has infinitely many minimizers $\v$ which are not
minimizers of $\mathcal E$, given by
$$\v=(\v^*-t\xx)\,\in  \argmin \mathcal F\setminus \argmin\mathcal E\qquad \hbox{if }\v^*\in \argmin \mathcal E,\ t>0\,.$$}
\end{example}
\vskip0.1cm
It is quite natural to ask whether condition \eqref{comp}, which is essential in the proof of Theorem \ref{mainth1}, may be dropped in order to obtain at least existence of $\min \mathcal F$: the answer is negative.\\
Indeed the next remark shows that, when compatibility inequality in \eqref{comp} is reversed
for at least one choice of the skew-symmetric matrix $\mathbf W$, then $\mathcal F$ is unbounded from below.
\begin{remark}\lab{controsegno} If
{\rm 
\beeq\lab{controsegnoeq}
\exists \,\mathbf W_*\in  \M^{N\times N}_{skew}\, :\qquad
\mathcal L(\z_{\mathbf W_*})>0\,, \quad \hbox{where }\z_{\mathbf W_*}=\frac  1 2 \mathbf W_*^2 \xx\,,
\eneq
 then
\beeq\lab{menoinf}
\inf_{\v\in H^1(\Omega;\R^N)} \mathcal F(\v)\ =\ -\infty.
\eneq
Indeed, by arguing as in \eqref{minFvsminE} 
we get
\begin{equation}\label{minFvsminF}
  \inf_{ H^1(\Omega;\R^N)} \mathcal F \ =\ \min_{H^1(\Omega;\R^N)} \mathcal E \ \,-
  \sup_{\mathbf W\in \M^{N\times N}_{skew}}
\mathcal L(\z_{\mathbf W}) \qquad \hbox{where }\ \z_{\mathbf W}=\frac 1 2  \mathbf W^2\x\,.
\end{equation}
Hence
\begin{equation*}
  \inf_{ H^1(\Omega;\R^N)} \mathcal F
  \ \leq \ \min_{H^1(\Omega;\R^N)} \mathcal  E \ \,-  \, \tau\mathcal L(\,\z_{\mathbf W_*}) \qquad \forall \,\tau>0\,,
\end{equation*}
which entails \eqref{menoinf}.}
\end{remark}%
Next example shows that in case of uniform compression along the whole boundary functional $\mathcal F$ is unbounded from below.
\begin{example} \lab{nocrit}
{\rm  Assume $\om\subset \R^N$ is a Lipschitz, connected open set$,\  N=2,3,\ \mathbf g\equiv \mathbf 0,\ \mathbf f= -\mathbf n,\ $ where $\n$ denotes the outer unit normal vector  to $\partial\om$
(see Fig.\ref{FigEx6.4}).
%
\vskip0.5cm
\begin{center}
\begin{tikzpicture}[scale=1.1]
 \begin{scope}
 %
 \draw [line width=0.08cm,red] [-] (-4.04,-1) -- (-0.96,-1);
 \draw [line width=0.08cm,red] [-] (-4.04,+2) -- (-0.96,+2);
 \draw [line width=0.08cm,red] [-] (-4.,-1) -- (-4.,+2);
 \draw [line width=0.08cm,red] [-] (-1.,-1) -- (-1.,+2);
 \end{scope}
\begin{scope}[very thick]
 \draw[->] (-3.5,2.5) -- (-3.5,2.1);
 \draw[->] (-3,2.5) -- (-3,2.1);
 \draw[->] (-2.5,2.5) -- (-2.5,2.1);
 \draw[->] (-2,2.5) -- (-2,2.1);
 \draw[->] (-1.5,2.5) -- (-1.5,2.1);
 \draw[->] (-3.5,-1.5) -- (-3.5,-1.1);
 \draw[->] (-3,-1.5) -- (-3,-1.1);
 \draw[->] (-2.5,-1.5) -- (-2.5,-1.1);
 \draw[->] (-2,-1.5) -- (-2,-1.1);
 \draw[->] (-1.5,-1.5) -- (-1.5,-1.1);
 \draw[->] (-4.5,-0.5) -- (-4.1,-0.5);
 \draw[->] (-4.5,0) -- (-4.1,0);
 \draw[->] (-4.5,0.5) -- (-4.1,0.5);
 \draw[->] (-4.5,1) -- (-4.1,1);
 \draw[->] (-4.5,1.5) -- (-4.1,1.5);
 \draw[->] (-0.5,-0.5) -- (-0.9,-0.5);
 \draw[->] (-0.5,0) -- (-0.9,0);
 \draw[->] (-0.5,0.5) -- (-0.9,0.5);
 \draw[->] (-0.5,1) -- (-0.9,1);
 \draw[->] (-0.5,1.5) -- (-0.9,1.5);
\end{scope}
\end{tikzpicture}
\end{center}
\vskip-0.4cm
\begin{figure}[h]
   \caption{Example \ref{nocrit}.}
   \label{FigEx6.4}
\end{figure}
Then \eqref{controsegnoeq} holds true hence, by Remark \ref{controsegno},
$\inf_{\v\in H^1(\Omega;\R^N)} \mathcal F(\v)\ =\ -\infty$. \vskip0.1cm Indeed, for every $ \mathbf W\in
\mathcal M^{N\times N}_{skew}$ such that $|\mathbf W|^2=2$ we obtain}
\\
$$
\int_{\partial \Omega} \!\mathbf f \cdot \mathbf W^2 \xx\, d\mathcal H ^{N-1}\, = \,
- \! \int_{\partial \Omega} \!\mathbf n \cdot \mathbf W^2 \xx \, d\mathcal H ^{N-1}\,= \,
- \!\int_{ \Omega} \dv  (\mathbf W^2 \xx) \,d\xx\,=\, -\, |\Omega|\,\tr\,  \mathbf W^2\,=\, 2\,|\Omega|>0\,.
$$
\vskip0.3cm
\rm{Therefore any Lipschitz open set turns out to be always unstable when uniformly compressed in the direction of the inward normal vector along its boundary and the linearized model proves inadequate
for this case even for small load.}
\end{example}

\end{document}